\documentclass[12pt,reqno]{amsart}
\usepackage{amssymb,amsmath,bm,bbm}
\usepackage{amsfonts,graphicx,amsthm,amstext,latexsym,amsfonts}
\usepackage{graphicx}
\usepackage{caption}
\usepackage{color}
\usepackage[usenames,dvipsnames]{xcolor}
\usepackage{tikz}
\usepackage{hyperref}
\usepackage[colorinlistoftodos]{todonotes}
\def\eps{\varepsilon}
\def\d{{\rm d}}

\def\R {\mathbb{R}}

\def\e {\varepsilon}

\def\z{\zeta}
\def\O{\Omega}

\def\bp{p}

\def\bbQ{\mathbb{Q}}
\def\bbY{\mathbb{Y}}

\newcommand{\FF}{{\mathcal F}}

\newcommand{\HH}{{\mathcal H}}

\newtheorem{proposition}{Proposition}[section]
\newtheorem{theorem}[proposition]{Theorem}

\newtheorem{lemma}[proposition]{Lemma}
\theoremstyle{definition}
\newtheorem{definition}[proposition]{Definition}
\newtheorem{remark}[proposition]{Remark}

\numberwithin{equation}{section}
\newcommand{\UUU}{\color{black}}

\newcommand{\EEE}{\color{black}}

\newcommand{\Nz}{{\mathbb N}}
\newcommand{\Rz}{{\mathbb R}}

\newcommand{\per}{\mathrm{Per}}

\newcommand{\cof}{\mathrm{cof\,}}

\newcommand{\MK}{\color{black}}

\newcommand{\EM}{\color{black}}

\newcommand{\KKK}{\color{black}}
\newcommand{\EOR}{\color{black}}
\newcommand{\md}{{\rm d}}
\newenvironment{proofth0}{\removelastskip\par\medskip   
\noindent{\bf Proof of {\rm {\bf Theorem \ref{Prop:diffuse-ex}}}.}
\rm}{\penalty-20\null\hfill$\square$\par\medbreak} 

\newenvironment{proofth1}{\removelastskip\par\medskip   
\noindent{\bf Proof of {\rm {\bf Theorem \ref{th1}}}.}
\rm}{\penalty-20\null\hfill$\square$\par\medbreak} 

\newenvironment{proofth2}{\removelastskip\par\medskip   
\noindent{\bf Proof of {\rm {\bf Theorem \ref{th2}}}.}
\rm}{\penalty-20\null\hfill$\square$\par\medbreak} 

\def \no#1#2#3 {{\bf #1} (#3), #2.}
\def \eds#1#2#3 {#1, #2, #3.}
\title[Equilibrium for multiphase solids with Eulerian
interfaces]
{Equilibrium for multiphase solids \\ with Eulerian
interfaces }

\author{Diego Grandi}
\address[Diego Grandi]{Dipartimento di Matematica e Informatica, 
  Universit\`a  degli Studi di Ferrara,
Via Machiavelli 30, 44121 - Ferrara, Italy.}
\email{diego.grandi@unife.it}

\author{Martin Kru\v{z}\'ik}
\address[Martin Kru\v{z}\'ik]{
Czech Academy of Sciences, Institute of Information Theory and Automation,
Pod vod\' arenskou ve\v z\' \i\ 4, 182 08, Prague 8, Czech Republic and Faculty of Civil Engineering, Czech Technical University, Th\'{a}kurova 7, 166 29, Prague 6, Czech Republic.}
\email{kruzik@utia.cas.cz}

\author{Edoardo Mainini}
\address[Edoardo Mainini]{Dipartimento di Ingegneria meccanica, energetica, gestionale e dei trasporti, 
  Universit\`a  degli studi di Genova, Via all'Opera Pia, 15 - 16145 Genova Italy.}
\email{mainini@dime.unige.it}

\author{Ulisse Stefanelli}
\address[Ulisse Stefanelli]{Faculty of Mathematics, University of
  Vienna, Oskar-Morgenstern-Platz 1, A-1090 Vienna, Austria,
Vienna Research Platform on Accelerating
  Photoreaction Discovery, University of Vienna, W\"ahringerstra\ss e 17, 1090 Wien, Austria,
 \& Istituto di
  Matematica Applicata e Tecnologie Informatiche {\it E. Magenes}, via
  Ferrata 1, I-27100 Pavia, Italy.
}
\email{ulisse.stefanelli@univie.ac.at}

\subjclass[2010]{\UUU 74G25, 
49J45,
}
\keywords{Elasticity, Eulerian-Lagrangian description, Phase transition, Variational methods, Gamma-convergence.}
\begin{document}
\begin{abstract}
\EM
We describe  a general phase-field model for hyperelastic multiphase materials. 
The model features an elastic energy functional that depends on the phase-field variable and a surface energy term that depends in turn on the elastic deformation, as it measures interfaces in the deformed configuration.
We prove existence of energy minimizing equilibrium states and $\Gamma$-convergence of diffuse-interface approximations to the 
sharp-interface limit. 

\EOR


\end{abstract}
\maketitle
%
%
%
\section{Introduction}
%
\MK Mathematical models of multi-component (or multi-phase)  materials
have \UUU attracted the \EEE attention of researchers for decades.  A
prominent example \UUU of multi-phase materials is provided by \EEE  shape memory
alloys, i.e., intermetallic materials having a high-temperature phase
called austenite and a low-temperature phase called martensite, \UUU
existing \EEE in many symmetry-related variants,  see \cite{ball-james,bhattacharya}. Mathematical analysis of elastostatic problems of such materials is involved because of the lack of suitable convexity properties. In fact, these materials exhibit complicated microstructures which are reflected in faster and faster oscillations of minimizing sequences driving the elastic energy functional to its infimum. Consequently, no minimizer generically exists and various methods have been developed to 
cope with this difficulty.  

\UUU A possibility to overcome the nonexistence issue   \EEE is to search for a lower
semicontinuous envelope of the energy functional that describes
macroscopic behaviour of the specimen \cite{dacorogna}. \UUU This \EEE
provides us with a \UUU solvable minimization \EEE problem and \UUU
ensures that \EEE every minimizer is reachable by a minimizing
sequence of the original problem. 
The downside of this method, called {\it relaxation}, is that \UUU
such envelope is usually not known \EEE closed form. 

\UUU A second \EEE option is to include a  higher-order deformation
gradient to the energy functional. In this case, we resort to
nonsimple materials, see e.g. \cite{ball-crooks, BCO, ball-mora, KM, Mu, Mu2} for
various attempts in this direction. Here, a convex function of the
second deformation gradient (strain gradient) penalizes spatial
changes of the first gradient, which \UUU introduces a second length
scale in the model and \EEE implies that oscillations in minimizing sequences have finite fineness.  
\EM Besides, some models that are discussed in the above contributions include  surface terms along the  discontinuity set of the first deformation gradient, see also \cite{F, P}. 	\KKK

A third option is the phase-field approach to multiphase materials, in which each phase of the material is identified by some  value of a suitable phase indicator. A surface energy is generally assigned to each phase-separating interface, 
which prevents
 repeated phase jumps at small scale, see for instance the general theory by \v{S}ilhav\'{y}  \cite{S2,S3}. 
In the gradient theory of phase transitions,  the surface area penalization  is relaxed by assuming that the change of phase takes place in a small but finite layer. This is the typical approach to the theory of Cahn-Hilliard fluids
 \cite{B, G, M, S},
the fundamental convergence result to the sharp interface limit \UUU
being \EEE established in \cite{M} \UUU based on \EEE the Modica-Mortola Theorem \cite{MM}. 

\UUU In this paper, we consider an elastic model for multi-phase
materials   inspired by \cite{S2,S3}. \UUU We introduce an \EEE energy functional depending on the first deformation gradient and a phase indicator
distinguishing particular material phases or variants. In particular, to each pair of continuous phases we
associate an interfacial energy, where interfaces are measured in the
deformed configuration. 
 \UUU  In fact, variational theories featuring Eulerian interfacial energy
terms can be traced back at least to
\cite{Gurtin75} and have been considered, for instance, in
\cite{Javili,Levitas10,Levitas14,Levitas14b}, among others.\EEE


In the particular case of a \UUU two-component \EEE material, a
diffuse-interface approximation to the \v{S}ilhav\'{y}'s model was \UUU
discussed \EEE in \cite{GKMS}. \UUU There, we \EEE proved that the
approximations $\Gamma$-converge to the sharp-interface  model. 
\EM
The aim of this paper is to extend that theory in several ways. We shall introduce a more general model,
allowing for a finite number of material components. For the treatment of this model, we shall further develop the analysis of {\it interfacial measures} from \cite{GKMS}, where a key role is played by the notion of mappings  of finite distorsion \cite{HK}. We shall also consider the borderline case of deformations in $W^{1,p}(\O;\R^n)$ with $p=n$,   hence without  requiring their H\"{o}lder continuity. 

%

\EEE
Let $\Omega\subset\mathbb{R}^n$, $n\ge 2$, be an open domain
representing the reference configuration of a multi-component
material. 
The composition of the material at each point is described by a \emph{component vector }   $z(x)=(z_1(x),\ldots, z_h(x))\in \mathbb R^h$. For instance, a mixture of $h$ chemical species can be described by the relative mass fraction $z_i\in [0,1]$,   of the $i^{\textrm{th}}$ component of the mixture for \MK $i=1,\ldots, h$.  If the components are immiscible, then at each point $x\in\O$ we have  $z_i(x)\in\{0,1\}$ and $z_i(x)=1$ if and only if  the material component $i$ is present at $x$. \EOR As a second example, we can  mention ferromagnetic materials, in which the spontaneous magnetization vector $z(x)\in\mathbb R^3$ can serve  as the component descriptor of the phase.

We introduce the discrete set  $P=\{p_{\alpha}\in \mathbb R^h\; |\;\alpha=1,\ldots,m \} $,  $m\ge 2$,  of  \emph{stable phases} characterized by the component vectors $z=p_\alpha$. The relation between the components number $h$ and the number of stable phases $m$ depends on the specific model. For instance, in an immiscible mixture with $h$ components, 
we may have  $m=h$ and  \EM    $ (p_{\alpha})_i=\delta_{i\alpha}$, where  $(p_\alpha)_i$ is  the $i^{\textrm{th}}$ component   of $p_\alpha$. On the other hand, \EOR if the component vector $z\in\mathbb R^3$ represents the (saturated) magnetization  vector of an anisotropic magnetic crystal, for instance with cubic anisotropy, one  needs to consider $m=2h=6$ stable phases corresponding to the six magnetization directions $\pm (1,0,0),\;\pm (0,1,0),\;\pm (0,0,1)$.

\textbf{Sharp interface model}.
\UUU In the sharp-interface setting, given a \EEE component-configuration field   $z:\O\to\mathbb R^h $ taking
values in $P$, \EEE  we  let $E_\alpha(z)\EOR:=\{x\in\O:z(x)=p_\alpha\}$, $\alpha=1,\ldots, m$. The sets  $(E_\alpha)_\alpha$  form a partition of  $\O$ describing  the spatial distribution of phases.
For a given deformation   $y:\Omega\to\O^y\subset \mathbb{R}^n$, we let $\zeta:\O^y\to \R^h$ denote the associated indicator function in the deformed configuration, i.e., $\zeta_i:=z_i\circ y^{-1}$, $i=1,\ldots, h$. 
 The set $E_\alpha^y:=y(E_\alpha)$ is the region occupied by phase $\alpha$ in the deformed configuration.

We  consider the stored energy functional for an elastic multiphase material
\begin{align}\label{F0}
\mathcal{F}_0(\zeta,y)=\int_\Omega W(\nabla y(x),\zeta(y(x)))\EOR\,dx+  \frac12\sum_{\alpha,\beta=1}^m d_{\alpha,\beta}\mathcal H^{n-1}(E_{\alpha,\beta}^y)
\end{align}
where $W$ is the stored bulk energy and 
\[
E^y_{\alpha,\beta}:=\partial^* E^y_\alpha\cap \partial^* E^y_\beta\cap \O^y
\]
 is the interface between $E_\alpha$ and $E_\beta$ in the deformed configuration.  Here, $\partial^*$ denotes the reduced boundary.  The coefficients $d_{\alpha,\beta}$ are suitable surface-tension parameters such that:  $d_{\alpha,\beta}=d_{\beta,\alpha}\ge 0$ and $d_{\alpha,\beta}=0$ if and only if $\alpha=\beta$.  The coefficients  are assumed to satisfy the following inequalities
\begin{align}\label{eq:tr-inequality}
 d_{\alpha,\beta}+d_{\beta,\gamma}\ge d_{\alpha,\gamma}
\end{align}
for any admissible triple of indexes $\alpha,\beta,\gamma$. 
This condition is necessary for lower semicontinuity of $\mathcal{F}_0$, see \cite{AB}. Indeed,  assume that  $d_{\alpha,\gamma}>d_{\alpha,\beta}+d_{\beta,\gamma}$ for some triple of phases and consider a sequence of  states where a layer of the phase $\beta$, of  thickness tending to zero, is inserted between the  layers $\alpha$ and $\gamma$. The bulk contribution in \eqref{F0} tends to the  value taken  in absence of the phase $\beta$ (the limit state); instead, the interfacial energy undergoes an increasing jump discontinuity in the limit process.    
The meaning of \eqref{eq:tr-inequality} also resides  in its relation with the notion of {\it separability} of interfaces from \cite{S2}, which would require the existence of coefficients $g_{\alpha}$, $\alpha=1,\ldots,m$, such that $d_{\alpha,\beta}=g_\alpha+g_\beta$ for any $\alpha$ and any $\beta$ between $1$ and $m$. The separability assumption implies \eqref{eq:tr-inequality} and the two are equivalent if $m=2,3$.    
\EOR


%
%

This model, as in \cite{B}, features a standard \EM sharp \KKK interface term for a multiphase material. On the other hand, the interface penalization is complemented by an elastic energy term that accounts for macroscopic deformation of the specimen, and the choice of taking the interface term in the deformed configuration is an example of interface polyconvex energy as described by \cite{S2,S3}.  

\textbf{Diffuse-interface model}.
 We are interested in providing a diffuse-interface approximation of the above energy. In a diffuse-interface model, 
 the phase field $z$   takes    values in \UUU $\mathbb R^h$. \EEE 
 The  phase-field   functional is  defined  as
 \[
  \mathcal{F}_{\eps}(\zeta,y)= \mathcal{F}^{\rm bulk}(\zeta,y)+  \mathcal{F}_{\eps}^{\rm int}(\zeta,y),
 \]
 where 
 \[
 \mathcal{F}^{\rm bulk}(\zeta,y):= \int_\Omega W(\nabla y(x),\zeta(y(x)))\,dx,\quad  \mathcal{F}_{\eps}^{\rm int}(\zeta,y):=\int_{\Omega
 ^y}\frac\eps2 |\nabla \zeta(\xi)|^2+\frac1\eps\,\Phi(\zeta(\xi))\,d\xi,
 \]
and where we denote by   $\xi$ (here and through the paper) the  variable in  deformed configuration,  i.e., $\xi\in \O^y$. 
 We have introduced a  continuous  multi-well potential $\Phi:\mathbb{R}^h\to\mathbb{R}^+$ with zeros only at $p_1,\ldots p_m$.
 The relationship between the two models is established by letting  \begin{equation}\label{dcoeff}d_{\alpha,\beta}:=d_\Phi(p_\alpha,p_\beta),\qquad \alpha=1,\ldots, m,\quad\beta=1,\ldots, m,\end{equation} where $d_\Phi$ the Riemaniann distance in $\mathbb{R}^h$ induced by $\sqrt{2\Phi}$, i.e.,
 $$ d_\Phi(p_\alpha,p_\beta):=\inf\left\{\int_0^1\sqrt{2\Phi(\gamma(t))}|\gamma'(t)|\,dt: \gamma\in C^1([0,1]; \mathbb R^h),\;\gamma(0)=p_\alpha,\;\gamma(1)=p_\beta\right\}.$$ 
 This   guarantees   symmetry, positivity, and \UUU the validity of \EEE the triangle inequality  \eqref{eq:tr-inequality} for the coefficients $d_{\alpha,\beta}$.

\MK
\subsection*{Plan of the paper}
 We first state our main results in
 Section~\ref{sec:mainresults}. \UUU In particular, we address \EEE
 the existence of minimizers for the diffuse-interface, as well as \UUU
 for \EEE the sharp-interface functionals $\mathcal{F}_{\eps}$ and
 $\mathcal{F}_{0}$ in Theorems~\ref{Prop:diffuse-ex} and \ref{th1},
 \UUU respectively. \EEE The approximation result is stated in
 Theorem~\ref{th2}. Properties of admissible deformations \EEE
 are reviewed in Section~\ref{sec:basics} \UUU whereas \EEE properties
 of interfacial measures, implicitly introduced in \cite{S2,S3},
 are \UUU detailed \EEE in Section~\ref{sil2}. 
 \UUU Our
 results mainly rest on proving a $\Gamma$-convergence
 statement. Indeed, a \EEE  proof of the fact that  $\mathcal{F}_{0}$
 is a lower bound for $\mathcal{F}_{\eps}$  is contained in
 Section~\ref{sec:liminf}. \UUU Eventually, \EEE proofs of the main theorems can be found in Section~\ref{sec:main}.

 \section{Main results}\label{sec:mainresults}
 \EM
 Let $\Omega\subset\mathbb R^n$, $n\ge 2$, be a bounded open Lipschitz set representing the reference configuration. In this section, we introduce the set of admissible couples $(y,\zeta)$ (deformation and phase indicator)  and we state the main results. \KKK
 
 \subsection{Admissible states} Following \cite{GKMS},
 we introduce the functional spaces of the admissible states.  For fixed  $q>n-1$ and $p\ge n$ (\EM not included in the notation for simplicity\KKK), we define the space of admissible deformations as \EOR
 \begin{equation} \label{eq:Y}
 \begin{aligned}
 \bbY:=\left\{y \in W^{1,p}(\O;\mathbb R^n)\ | \ \det\nabla y>0 \text{ a.e.}\, , 
 \int_\O\det\nabla y(x)\,\md x\le |\O^y|,\, 
 K_y\in L^q(\Omega)
 \right\}
 \end{aligned}
 \end{equation}
   Here, $K_y$ denotes the optimal distorsion function associated to the deformation map $y$, see Definition \ref{finitedistorsion} below.  Any element of $\mathbb Y$ has  a continuous representative  which is a homeomorphism.  This is a consequence of the Ciarlet-Ne\v{c}as \cite{ciarlet-necas} condition appearing in \eqref{eq:Y} and of the $L^q$ integrability of the distorsion function as shown in 
   \cite{GKMS} for $n=3$. \EOR The arguments therein straightforwardly apply for any dimension $n\ge 2$. Later in section \ref{sec:basics} we shall derive more properties of the set of admissible deformations.
 
 Recalling that $P\subset \R^h$ is the finite set of stable phases, \KKK
 we define the sets of the states, including the \EM states \KKK for the sharp interface model 
  $$
 \mathbb{Q}:= 
 \{(y,\zeta)\ | \ y\in \mathbb{Y}, \ \zeta\in BV(\Omega^y;\R^h), \ \z(\xi)\in P \ \mbox{for a.e. } \xi\in\O^y\},
 $$ 
 \KKK and  for the diffuse interface model 
  $$
 \widetilde{\mathbb{Q}}^R:= 
 \{(y,\zeta)\ | \ y\in \mathbb{Y}, \ \zeta\in W^{1,2}(\Omega^y;\R^h), \ |\z(\xi)|\le R \ \mbox{for a.e. } \xi\in\O^y\},
 $$ 
 where $R>0$. 
 A natural compatibility condition for the two models is  $R>\max_{\alpha\in\{1,\ldots, m\}}|p_\alpha|$,
 \EM so that for a couple $(y,\zeta)\in\widetilde {\mathbb Q}^R$, $\zeta$ may take values in $P$.   
 
 %
 Letting
 $\Gamma_0\subset\partial\O$ be  relatively open in $\partial \O$
 with 
 $\HH^{n-1}(\Gamma_0)>0$,  and letting $y_0\in\mathbb Y$ be continuous up to $\partial\Omega$, \EOR
 we introduce the associated function spaces with  Dirichlet boundary conditions 
 $${{\mathbb{Q}}}_{(y_0,\Gamma_0)}:= \{(y,\z) \in
 \mathbb{Q} \ | \
 y=y_0\text{ on $\Gamma_0$}\}, \qquad {\widetilde{\mathbb{Q}}}^R_{(y_0,\Gamma_0)}:= \{(y,\z)\in \widetilde{\mathbb{Q}} \ | \,
	y=y_0\text{ on $\Gamma_0$}\},$$
\EM where the relation $y=y_0$ on $\Gamma_0$ is understood in the sense of traces. Moreover,  $y_0$  is required to be nonconstant on $\Gamma_0$ (i.e. $\Gamma_0$ does not shrink to a point).  \KKK
 We further define $\mathbb Q_{y_0}:=\mathbb Q_{(y_0,\partial\O)}$ and $\widetilde{\mathbb Q}^R_{y_0}:=\widetilde{\mathbb Q}^R_{(y_0,\partial\O)}$. \KKK

 Given $y_0\in \mathbb Y$, the compatibility between the boundary condition $y=y_0$ on $\Gamma_0$ and the choice of the energy functional is enforced by assuming that 
 \begin{equation}\label{statofinito}
 \mbox{there exists} \ \ (y,\z)\in\bbQ_{(y_0,\Gamma_0)} \ \ \text{such that} \ \ \mathcal{F}_{0}(y,\z)<\infty.
 \end{equation}
 

\EOR
\subsection{The elastic energy}
The elastic energy, both \UUU in the diffuse- and in the sharp-interface
case, is given by \EEE the bulk integral functional
$
\mathcal{F}^{ \mathrm{bulk}}(\zeta,y)
$.
The following assumptions are made for the energy density \EM $W:\mathbb R^{n\times n}\times\mathbb R\to(-\infty,+\infty]$.  \KKK
\begin{equation}
\begin{aligned}\label{lscassumption}
&\text{The map $ W(\cdot,\cdot)$ is lower semicontinuous in $
  \mathbb{R}^{n\times n}\times \mathbb R^h$},\\
&\mbox{for any  $z\in \mathbb R^h$, \KKK the map $F\mapsto W(F, z) $ is {\it polyconvex}, }\\
& W(RF,z) = W(F,z)\quad \forall R\in{\rm SO}(n),\;\;\forall F  \in \R^{n\times n},\;\; \forall z\in\mathbb R^h,  
 \end{aligned}
 \end{equation}
\EM where ${\rm SO}(n)$ appearing in the standard frame-indifference property is the special orthogonal group, i.e., ${\rm
 	SO}(n)=\{R\in  \R^{n\times n} \ | \ R R^T =  I, \ \det
 R=1\}$. \KKK
 The notion of 
 { polyconvexity} \cite{Ball-1977} requires that the map $F\mapsto W(F, z) $
  can be written as a convex function of all of the minors (subdeterminants) of $F$. For instance, if $n=3$,
 \begin{align*}  %
 W(F,z ):=
 \begin{cases}
 w(F, \cof F, \det F, z) & \mbox{ if } \det F >0, \\
 \infty \mbox{ otherwise}
 \end{cases}
 \end{align*}
 for a convex function $w(\cdot,z):\R^{19} \to \R$, at \UUU all \EEE
 $z\in \mathbb R^h$, \KKK where $\cof F$ denotes the cofactor matrix
 of $F$. \EM We further assume that $W(\cdot, z)$ satisfies a suitable coercivity property. \KKK More precisely, we require that  
  there exists \EOR  $C>0$, \UUU $ p\ge n$, $ r>1$, and $ q>n-1$ \EEE such that
 \begin{align}
 & \UUU W(F,z)\ge C\left(|F |^p+(\det F)^r +\frac{|F |^{nq}}{(\det F
   )^{q}}  \right)-\frac{1}{C} \qquad \forall F  \in \R^{n \times n}, \EM \;\;\forall z\in\mathbb R^h. \label{W-ass2} 
 \end{align}
  The third term on the right-hand side of \eqref{W-ass2} ensures that
 deformation  gradients  $F=\nabla y$ with finite energy  will
 have a $q$-integrable distorsion function  $F\mapsto
 |F|^n/\det F$. Notice that $F\mapsto|F|^n/\det F$ is polyconvex on the
 set of matrices with positive determinant. \EM On the other hand, we mention that it is possible to drop the restriction $W(F,z)\ge C(\det F)^r$ in case $p>n$. \KKK
  
 A typical example of a bulk energy functional $W$ is 
 \begin{equation}\label{example}
 W(F,z)=\EM \sum_{i=1}^{h}  z_i^+ W_i(F) +
 \Big(1-(z_1 + \dots +z_{h})\Big)^+ W_{h+1}(F)\EEE 
 \end{equation}
 where we assume that  the listed properties \eqref{lscassumption}-\eqref{W-ass2} are \UUU
 uniformly \EEE satisfied by every elastic potential $W_i$ at the place of $W$. \UUU The
 latter corresponds to a mixture ansatz, where notation is
 prepared for the general case $z\in \mathbb R^h$ of the phase-field
 approximation.  In the sharp-interface case, we have that $z\in P$, \EM where the set $P$ of stable phases includes the origin and the standard orthonormal basis of $\mathbb R^h$ (thus, $m=h+1$)   \KKK and
 the
latter elastic energy takes the classical form
\[
 W(F,z)=\EM \sum_{i=1}^{h}  z_i W_i(F)+(1-(z_1 + \dots +z_{h}))\, W_{h+1}(F),\qquad  z\in P.
\]
We also note that the assumptions
\eqref{lscassumption}-\eqref{W-ass2} on $W$ could be imposed in the
physical case  $z\in {\rm Conv}
(P)$ first, and then extended to the whole $\mathbb R^h$ by a suitable
projection construction. \EEE

\subsection{Statement of the main results}
\UUU Owing to the above-introduced notation, we are now in the
position of stating the main results of the paper. \EOR
\EM In the next three statements, the following  underlying assumptions are understood to hold. $\Omega$ is a bounded open Lipschitz domain.  The exponents $p,q$ in the definition of the set of admissible deformations $\mathbb Y$ are of course given by assumption \eqref{W-ass2}. 
As discussed in the introduction, 
the multiwell potential $\Phi:\mathbb R^h\to\mathbb R^+$ is continuous and vanishing only at points of $P$, and the coefficients $d_{\alpha,\beta}$ appearing in \eqref{F0} are given by \eqref{dcoeff}. About the Dirichlet datum, we require that $y_0\in\mathbb Y$ is continuous up to the boundary of $\Omega$ and not constant on  $\Gamma_0$. Here, $\Gamma_0\subset\partial\O$ is  relatively open in $\partial \O$
 and
 $\HH^{n-1}(\Gamma_0)>0$.

\KKK
\begin{theorem}[\UUU Existence for the \EM diffuse-interface model\KKK]\label{Prop:diffuse-ex}
	 Let $\eps>0$  and $R>0$  be fixed. Suppose that  $(y,\zeta)\in\widetilde{\mathbb Q}^R_{(y_0,\Gamma_0)}$  exists such that $\FF_\eps(y,\zeta)<\infty$. \EOR
	 Let assumptions \eqref{lscassumption}, \eqref{W-ass2} hold.  Then,  there is a minimizer of $\FF_\eps$ on ${\widetilde{\mathbb{Q}}}^R_{(y_0,\Gamma_0)}$.
\end{theorem}

\begin{theorem}[\UUU Existence for the sharp-interface model]\label{th1} 
Under assumptions \eqref{statofinito}, \eqref{lscassumption},  \eqref{W-ass2},
	the functional $\mathcal{F}_0$ admits a minimizer on
	${{\mathbb{Q}}}_{(y_0,\Gamma_0)}$. 
\end{theorem}

The third main result  states that \UUU the phase-field 
indeed approximates the sharp-interphase model, namely \EEE
$\mathcal{F}_0$ is \UUU the \EEE $\Gamma$-limit \cite{DalMaso93} of
the family $(\mathcal{F}_{\eps})_\eps$. It requires an  additional
assumption on the boundary datum  $y_0$. Namely,   we ask for \EEE
$\Gamma_0=\partial\Omega$, \UUU i.e., Dirichlet conditions are
imposed on the whole boundary, and $\O^{y_0}$ is assumed
to be  a Lipschitz domain. Moreover, assumptions on $W$ have to
be strengthened by additionally asking \EEE
\begin{equation}\label{W-ass1}
\mbox{  the map $W(F,\cdot):\mathbb R^h\to\mathbb R$ is continuous for any $F\in \mathbb R^{n\times n}$.
}
\end{equation}
 We shall also require that for any $R>0$, given $y\in\mathbb Y$ and given $z\in L^1(\O;\mathbb R^h)$ such that $|z|\le R$ a.e. in $\O$, there holds \begin{equation}\label{W-ass5}
 \int_\O W(y(x),z(x))\, dx <\infty\;\;\Rightarrow \;\; \int_\O \sup_{\{z\in\mathbb R^h: |z|\le R\}}W(y(x),z) \,dx <\infty. 
 \end{equation}\KKK
 \EM When considering the mixture example \eqref{example}, the latter assumption is satisfied under the following comparability condition between the elastic potentials of the different phases: if $y\in\mathbb Y$ is such that  $W_i(\nabla y)$ is integrable on $\Omega$ for some $i=1,\ldots,m$, then  $W_j(\nabla y)$ is integrable on $\Omega$ for any $j\neq i$. 
 \KKK

\begin{theorem}[Phase-field approximation]\label{th2}
Let assumptions \eqref{statofinito}, \eqref{lscassumption}, \eqref{W-ass2}, \eqref{W-ass1}  and \eqref{W-ass5} hold. 
	Let $y_0\in\mathbb Y$ be such that $\Omega^{y_0}$ is a Lipschitz domain. 
 There exists $R_0>0$ such that if $R>R_0$ the following holds. 	For every vanishing sequence  $(\eps_k)_k $ of positive numbers  and every sequence
	$(y_k,\zeta_k)_k$ of minimizers of $\mathcal{F}_{\eps_k}$  on 
	$\widetilde\bbQ_{y_0}^R$, there exists $(y, \z)\in \bbQ_{y_0}$ such
	that, up to  not relabeled subsequences, 
	\begin{itemize}
		\item[i)] $y_k\to y$ weakly in $W^{1,p}(\O;\R^n)$ as $k\to\infty$ 
		\item[ii)]  $\zeta_k \to{\zeta}$ strongly in ${L^1(\Omega^y; \R^h}) $ as $k\to\infty$
		\item[iii)] $(y,\z)$ minimizes $\mathcal{F}_0$ on $\bbQ_{y_0}$.
	\end{itemize}
\end{theorem} 


\UUU
\begin{remark}[Incompressibility]\label{rem}
The above results can be specialized to the case of an incompressible
material. 
Indeed, one could impose \EM the incompressibility  constraint by letting $W(F,z)=+\infty$ if $\det F\neq1$,
which is compatible with the assumptions on $W$. For the model case \eqref{example} one might require $W_\alpha(F)=+\infty$ if $\det F\neq 1$ for any $\alpha=1,\ldots, m$. \KKK
\end{remark}

\EOR

\begin{remark}[Mass constraint]\label{lastremark}\rm \UUU Our analysis would allow
  additionally imposing the constraint \EEE
\begin{equation*}\label{masse}
\int_{\O^y}\zeta_i(\xi)\,d\xi=\int_\O \zeta_i(y(x))\,\det\nabla y(x)\,dx=M_i, \qquad i=1,\ldots, h
\end{equation*}
for \UUU  given values  $M_i$. By interpreting $\zeta_i$ as volume
densities, the latter corresponds to constraining the mass of the
single phases. In the incompressible case, see Remark \ref{rem}, such
constraints can be equivalently rewritten, for couples $(y,\zeta)$ with finite energy, in the more standard form \EEE
$$\int_\O z_i(x)\,dx=\int_\O \zeta_i(y(x))\,dx=M_i\qquad i=1,\ldots, h.$$

\end{remark}
\KKK

\section{\EM Properties of admissible deformations\KKK}\label{sec:basics}
%

\EM In this section we introduce the notion of mappings of finite distorsion and the distorsion function which appears in the definition \eqref{eq:Y} of the set $\mathbb Y$ of admissible deformations. Based on the properties of such mappings, for which we mostly refer to \cite{HK}, we shall obtain a suitable closure property of $\mathbb Y$.
Let us start by some basic definitions. In this section, $\Omega$ is an arbitrary open set of $\mathbb R^n$. \KKK

 The set of finite Radon measures $\mu$ on $\O$ with value in $\mathbb R^n$ is denoted by $\mathcal M(\O;\mathbb R^n)$ and it is normed by the total variation
\[
|\mu|(\O):=\sup\left\{\int_\O f\cdot\,d\mu\;|\; f\in C^0_c(\O;\mathbb R^n),\; \|f\|_\infty\le 1\right\}.
\]  
The weak convergence in $\mathcal M(\O;\mathbb R^n)$ of a sequence $(\mu_n)\subset \mathcal M(\O;\mathbb R^n)$ to $\mu\in \mathcal M(\O;\mathbb R^n)$   is defined by
\[
\int_\O f\cdot d\mu_n\to\int_\O f\cdot\,d\mu \quad \mbox{for any $f\in C^0_c(\O;\mathbb R^n)$}.
\]
 For a measurable set $E \subset\Omega$, we
denote the $n$-dimensional Lebesgue measure by $|E|$ and the
$m$-dimensional Hausdorff measure by $\HH^m(E)$.  By   $\chi_E$ we
denote the characteristic function of $E$.  If $g\in L^1_{loc}(\O)$,
 we say that $g\in BV(\O)$ if 
$$
|\nabla g|(\O):=\sup\left\{\int_{\O}g\,\textrm{div}\varphi\, \d x\ | \ \varphi\in C^\infty_{\rm c}(\O;\R^n),\;\|\varphi\|_\infty\leq 
1\right\}<+\infty,
$$
and we say that a measurable set $E\subset \O$ is a set of finite perimeter in $\O$ if $\chi_E\in BV(\O)$. We use the notation $\per (E,\O):=|\nabla \chi_E|(\O)$.
\EM For a set of finite perimeter $E$ in $\Omega$, there is a subset $\partial^*E $ of $\partial E$ (called {\it reduced boundary}) such that $\per(E,\Omega)=\HH^{n-1}(\partial^*E\cap\Omega)$, see \cite{AFP}.
\KKK
 Given  $y: \O \to \Rz^n$, we will use the notations $\O^y:=y(\O)$ and $E^y := y(E)$, \EM and we recall that $y$ is said to satisfy the Lusin condition $N$ if $|E|=0\Rightarrow |E^y|=0$. \KKK
 


\begin{definition}[Finite distorsion]\label{finitedistorsion} 
 Let $\O\subset\mathbb{R}^n$  for $n\ge 2$  be an open set. 
A Sobolev map $y\in W^{1,1}_{\rm loc}(\Omega;\R^{ n})$ with $\det\nabla y\ge 0$ almost everywhere in $\Omega$ is said to be of {\it finite distorsion} if $\det\nabla y\in L^1_{\rm loc}(\Omega)$ and  there is  a function $K:\O\to[1,+\infty]$ with $K<+\infty$ almost everywhere in $\O$ such that $|\nabla y|^n\le K\det\nabla y$. 
For a mapping $y$ of finite distorsion, the {\it (optimal) distorsion function} $K_y:\Omega\to\mathbb{R}$ is defined as
\MK
\[
K_y(x):=\left\{\begin{array}{ll}{|\nabla y(x)|^{ n}}/{\det\nabla y(x)}\quad&\mbox{if $\det\nabla y(x)\neq0$,}\\
1\quad&\mbox{if $\det\nabla y(x)=0$}.\end{array}\right.
\]
\end{definition}

\EM
The following result is a closure property of the set of admissible
deformations. \EEE

\begin{lemma}[Closure]\label{init} Let $p\ge n$ and let $q>n-1$.
Let  $y\in W^{1,p}(\Omega;\mathbb R^n)$ and
let $(y_k)_k\subset \mathbb Y$ be a sequence  such that
\begin{itemize}
\item[i)] $y$ is not constant
\item[ii)] $y_k\to y$ weakly in $W^{1,p}(\Omega;\mathbb R^n)$ as $k\to\infty$, 
  \item[iii)] $C:=\sup_{k\in\mathbb N}\|K _{y_k}\|_{L^q(\Omega)}<+\infty$.
  \end{itemize}
  Then  $y\in \mathbb Y$. In particular, $y$ has a  continuous
  representative which is a homeomorphism.
\end{lemma}

%
%



\begin{proof}
It is enough to consider the hardest case $p=n$. We recall from \cite[Section 3]{GKMS} that any element of $\mathbb Y$ has a continuous representative which is a \UUU homeomorphisms \EEE of $\Omega$ onto~$\Omega^y$.

First of all, up to extraction of a not relabeled subsequence, there exists a function $K\in L^q(\Omega)$ such that $K_{y_k}\to K$ weakly in $L^q(\Omega)$ as $k\to\infty$. Then, \EM a result by Gehring and Iwaniec \cite{GI}, see also \cite{FMS} \KKK ensures that $y$ is a mapping of finite distorsion such that 
$$\|K_y\|_{L^q(\Omega)}\le\|K\|_{L^q(\Omega)}\le\liminf_{k\to+\infty}\|K_{y_k}\|_{L^q(\Omega)}\le C.$$ In particular, $y$ has a continuous representative by \cite[Theorem 2.3]{HK}. Moreover, since $y_k\to y$ weakly in $W^{1,n}(\Omega;\mathbb R^n)$,  the higher integrability result by M\"uller \cite{Mu} entails $\det\nabla y^k\to\det\nabla y$ weakly in $L^1(E)$ for any open set $E$ compactly contained in $\Omega$.
Therefore, we may invoke the result in \cite[Theorem 4.4]{GP} to infer that $|E^{y_k}|\to |E^y|$ as $k\to+\infty$, recalling that the measure-theoretic images from \cite{GP} are in this case reduced to the usual images through the continuous representatives of $y_k$ and $y$. Moreover, \EM by \cite[Theorem 4.5]{HK},  continuous representatives of $W^{1,n}(\Omega;\mathbb R^n)$ mappings of finite distorsion satisfy the Lusin condition $N$, and thus the area formula holds with equality, see \cite[Theorem A.35]{HK}. In particular,  since the $y_k$'s are in fact  homeomorphisms, the area formula yields \KKK
\[
\int_E\det\nabla y\,dx=\lim_{k\to+\infty}\int_E \det\nabla y_k\,dx= \lim_{k\to+\infty}|E^{y_k}|=|E^y|.
\] 
By taking now an increasing sequence of open sets  $E_j$, compactly contained in $\Omega$, such that $\cup_{j=1}^{\infty} E_j=\Omega$, and by applying the monotone convergence theorem, we obtain the validity of the \UUU Ciarlet-Ne\v cas \EEE condition (with equality) for $y$. We notice that since $y$ is not constant, it is an open map by \cite[Theorem 3.4]{HK}, therefore $\O^y$ is open.  The \UUU Ciarlet-Ne\v cas \EEE condition entails that the multiplicity function $N(\Omega,y,\cdot)$ of $y$ on $\Omega$ is a.e. equal to $1$ in $\Omega^y$: indeed, since $y$ satisfies the Lusin condition $N$, the area formula and the Ciarlet-Ne\v{c}as condition yield 
\[
|\Omega^y|\le \int_{\Omega^y}N(\Omega,y,\xi)\,d\xi=\int_\Omega\det\nabla y\le |\Omega^y|
\]
so that $N(\Omega,y,\xi)=1$ for a.e. $\xi\in\Omega^y$.
 By invoking \cite[Lemma 4.13]{HK} we conclude that $\det\nabla y>0$ a.e. in $\Omega$. This proves that $y\in\mathbb Y$. 
\end{proof}

\KKK

\section{\EM Interfacial measures \KKK}\label{sil2}
%

\EM This section is devoted to introduce a fundamental notions of our theory, in particular we introduce {\it interfacial measures}  and \EEE provide
a generalization of \cite[Theorem 2.2]{GKMS}. 
 \EM In this section,  $\O\subset\mathbb{R}^n$ denotes a generic open set. \KKK

\begin{definition}[Interfacial measure]\label{ifmeas}
 Let $p\ge n$. 
Given a homeomorphism of finite distorsion $y\in W^{1,p}_{loc}(\O;\mathbb R^n)$ and $g\in L^r_{loc}(\Omega)$ for some $r\in[\tfrac{p}{p-n},+\infty]$, we say that $p_{y,g}\in\mathcal M(\O;\mathbb R^n)$ is an interfacial measure  for the couple $(y,g)$ \KKK if 
\begin{equation}\label{eq:p_yg-1}
 \int_{ \Omega} g\,\cof(\nabla 
y):\nabla\psi \, \d x=\int_{\O}\psi\cdot \d \bp_{y,g}\qquad \mbox{for any } \psi\in C^\infty_{\rm c}(\O;\R^n).
\end{equation}
\end{definition}
The relevance of this notion comes from its role in the characterization of interface areas in the deformed configurations, in case $g$ is a distance function from an energy well. It  will be thoroughly discussed in Theorem \ref{reverse} and in the rest of the paper. If $y$ is the identity map on $\O$, requiring the existence of an interfacial measure is equivalent to saying that $g\in BV(\O)$. If \eqref{eq:p_yg-1} holds, $p_{y,g}$ is the distributional divergence of $-g\,\cof\nabla y$ in $\O$.

\UUU In the following, we \EEE   give a characterization of those couples $(g,y)$, where  $g\in L^\infty_{loc}(\O)$ \KKK and $y$ is a  homeomorphism  in  $W^{1,n}_{loc}(\O)$,  such that $g\circ y^{-1}\in BV(\Omega^y)$. We state the theorem after having  introduced some preliminary notation.

	For a homeomorphism  $y:\Omega\to\mathbb R^n$ and a finite Radon measure
	 $\mu\in\mathcal M(\Omega^y;\mathbb R^n)$, the pull-back measure of $\mu$ through  $y$, denoted $y_\flat\mu$, is the measure in $\mathcal M(\Omega;\mathbb R^n)$ defined by
	\[
	\int_\Omega \psi\cdot\,d(y_\flat\mu)=\int_{\Omega^y}\psi\circ y^{-1}\cdot\,d\mu\qquad\mbox{for any bounded Borel function $\psi:\Omega\to\mathbb R^n$}.
	\] 
	Clearly, $y_\flat \mu(\Omega)=\mu(\O^y)$. Moreover, $|y_\flat\mu|(\Omega)= |\mu|(\Omega^y)$, since is $y$ is a homeomorphism.

	\begin{theorem}[\UUU Characterization] \label{reverse}
	\MK Let  $p\ge n$ and let  $y\in
W^{1,p}_{\rm loc}(\Omega;\R^n)$ be a homeomorphism of finite distorsion. 
Let $g\in L^\infty_{loc}(\Omega)$. \EEE
 Then, $g\circ y^{-1}\in BV(\Omega^y)$ if and only if a finite Radon measure $\bp_{y,g}\in\mathcal{M}_{}(\Omega;\R^n)$ exists such that \eqref{eq:p_yg-1} holds.
In such case, 
\begin{equation}\label{emme}
\bp_{y,g}=y_\flat(\nabla(g\circ y^{-1}))=-\mathrm{div}(g\,\cof\nabla y)\qquad\mbox{in $\;\mathcal M(\Omega;\mathbb R^n)$}.
\end{equation}
	\end{theorem}
\begin{proof}
We preliminarily observe that a homeomorphism  in $W^{1,n}_{\rm loc}(\Omega;\R^n)$
 satisfies the Lusin's condition $N$ \cite[Theorem. 3]{reshetnyak}, i.e., $|E|=0\Rightarrow |E^y|=0$ for any measurable set $E\in\Omega$. As a consequence  
$E^y$ is measurable for any measurable set $E\in\Omega$ and we may apply the area formula, see \cite[Theorem A.35]{HK}:  if $f\in L^{r}_{loc}(\Omega)$ for some $r\in[\tfrac{p}{p-n},+\infty]$,  for the measurable function $f\circ y^{-1}$ there holds
\[
\int_{E^y}|f|\circ y^{-1}\,d\xi=\int_E |f|\,\det\nabla y\,dx,
\]
for any measurable set $E\subset\Omega$.
In particular  we obtain $f\circ y^{-1}\in L^1_{loc}(\Omega^y)$, since we have by assumption $\det\nabla y\in L^{p/n}_{loc}(\Omega)$ and $f\in L^{r}_{loc}(\Omega)$.
The Lusin condition $N$ also implies that $\|g\circ y^{-1}\|_{L^\infty(E^y)}=\|g\|_{L^\infty(E)}$ for any measurable set $E\subset\O$ so that we obtain $g\circ y^{-1}\in L^\infty_{loc}(\O^y)$, since $g\in L^\infty_{loc}(\O)$.
\KKK

{\textbf{Step 1.} }
Let us assume $g\circ y^{-1}\in BV(\Omega^y)$. We shall verify that, by taking $\bp_{y,g}:=y_\flat(\nabla(g\circ y^{-1}))$,   \eqref{eq:p_yg-1} holds along with \eqref{emme}. 

First, we observe that $y_\flat(\nabla(g\circ y^{-1}))\in\mathcal M(\O;\mathbb R^n)$ by definition of pull-back, since $\nabla(g\circ y^{-1})\in\mathcal M(\O^y;\mathbb R^n)$.
Let $\psi\in C^\infty_c(\Omega;\mathbb R^n)$. Let $G_\eps:=(g\circ y^{-1})\ast\rho_\eps$, where $\rho_\eps(x):=\eps^{-n}\rho(x/\eps)$, $x\in\mathbb R^n$, and $\rho$ is the standard unit symmetric mollifier in $\mathbb R^n$, so that  
 (up to passing to a vanishing sequence, which we do not include in the notation)
 $G_\eps\to g\circ y^{-1}$ a.e. in $\Omega^y$ and $\nabla G_\eps\rightharpoonup \nabla(g\circ y^{-1})$ weakly in $\mathcal M(\Omega^y;\mathbb R^n)$. 
Therefore,
\begin{equation}\label{unos}
\int_\Omega \psi\cdot d(y_\flat(\nabla(g\circ y^{-1})))=\int_{\Omega^y}(\psi\circ y^{-1})\cdot\,d(\nabla(g\circ y^{-1}))=\lim_{\eps\to 0}\int_{\Omega^y}(\psi\circ y^{-1})\cdot \nabla G_\eps\,d\xi. 
\end{equation}
There holds
 $(\nabla y)^{-T}\nabla (G_\eps\circ y)= (\nabla G_\eps)\circ y$  a.e. in $D:=\{x\in\Omega:\det\nabla y(x)>0\}$.
 The cofactor matrix is divergence-free,
 implying   $\mathrm{div}((\cof\nabla y)^T\psi)=\cof\nabla y:\nabla \psi$. Moreover, $\cof\nabla y=0$ holds a.e. on $\Omega\setminus D$ since $y$ is a mapping of finite distorsion. Hence,
\begin{equation}\label{dues}\begin{aligned}
&\int_{\Omega^y}(\psi\circ y^{-1})\cdot \nabla G_\eps\,d\xi =\int_{D}(\det\nabla y)	\,\psi\cdot(\nabla G_\eps)\circ y\,dx\\&\qquad=
\int_{D}(\det\nabla y)	\,\psi\cdot (\nabla y)^{-T}\nabla(G_\eps\circ y)\,dx=
\int_{D}(\det\nabla y)	\,(\nabla y)^{-1}\psi\cdot\nabla (G_\eps\circ y)\,dx\\&\qquad=
-\int_\Omega (G_\eps\circ y)\,\mathrm{div}((\cof\nabla y)^T\,\psi)\,dx =
-\int_\Omega (G_\eps\circ y) \,\cof\nabla y:\nabla \psi\,dx
\end{aligned}\end{equation}
Since $G_\eps\to g\circ y^{-1}$ pointwise a.e. in $\Omega_y$, we obtain $G_\eps\circ y\to g$ a.e. in $D$.
Indeed, the area formula again implies that for a measurable set $E\subset D$ there holds
$
|E^y|=\int_E\det\nabla y
$ so that $|E^y|=0$ implies $|E|=0$. 
In particular, if $E=D\cap\mathrm{supp}(\psi)$, then $\|G_\eps\circ y\|_{L^\infty(E)}=\|G_\eps\|_{L^\infty(E^y)}\le \|g\circ y^{-1}\|_{L^\infty(E^y)}=\|g\|_{L^{\infty}(E)}<\infty$.
\KKK
 As $\cof\nabla y\in L^1_{loc}(\Omega) $ and $\cof\nabla y=0$  a.e. on $\Omega\setminus D$, by dominated convergence we obtain
\begin{equation}\label{tres}
\lim_{\eps\to 0}\int_\Omega (G_\eps\circ y)\,\nabla \psi:\cof\nabla y\,dx=\int_\Omega g\,\nabla \psi:\cof\nabla y\,dx.
\end{equation}
By combining \eqref{unos}, \eqref{dues} and \eqref{tres} we get
\[
\int_\Omega \psi\cdot d(y^{-1}_\flat(\nabla(g\circ y^{-1})))= -\int_\Omega g\,\nabla \psi:\cof\nabla y\,dx
\]
for any $\psi\in C^\infty_c(\Omega;\mathbb R^n)$. Hence,  $\bp_{y,g}$ satisfies \eqref{eq:p_yg-1} and \eqref{emme} holds.

{\textbf{Step 2.}}
Let us now assume that $\bp_{y,g}\in \mathcal M(\Omega;\mathbb R^n)$ exists such that \eqref{eq:p_yg-1} holds and let us verify that $g\circ y^{-1}\in BV(\Omega^y)$.

The area formula gives
\begin{equation}\label{eq:chain}\begin{aligned}
&|\nabla(g\circ y^{-1})|(\Omega^y)=\sup\left\{\int_{\Omega^y}g(y^{-1}(\xi))\,\textrm{div}\varphi(\xi)\, \d \xi \;|\;\varphi\in C^\infty_{\rm c}(\O^y;\R^n),\;\|\varphi\|_\infty\leq 
1\right\}\\&\qquad
=\sup\left\{\int_{\O}g(x)\,\textrm{div}\varphi(y(x))\,\textrm{det}\nabla y(x)\,\d 
x\;|\;\varphi\in C^\infty_{\rm c}(\O^y;\R^n),\;\|\varphi\|_\infty\leq 1\right\}\\&\qquad
=\sup\left\{\int_{\O}g\,\cof(\nabla y):\nabla(\varphi\circ y)\,\d x\;|\;\varphi\in 
C^\infty_{\rm c}(\O^y;\R^3),\;\|\varphi\|_\infty\leq 1\right\},
\end{aligned}
\end{equation}
 where the second equality is due to the
identity $(\textrm{div}\varphi)\circ
  y\det\nabla y=\cof\nabla y:\nabla(\varphi\circ y)$
which holds a.e. in $\Omega$, as a consequence of the chain rule
and  of  the matrix identity $(\cof {A}) {A}^T=I\det {A}$.
 As  $y\in W^{1,p}_{\rm loc}(\Omega;\R^n)$,  we have $\cof\nabla y\in
L^{q}_{\rm loc}(\Omega)$  with  $q=p/(n-1)$. Since $g\in L^{{p}/{(p-n)}}_{loc}(\Omega)$ we get $g\,\cof\nabla y:\nabla(\varphi\circ y) \in L^1_{loc}(\Omega)$. 
The function $g\,\cof\nabla y:\nabla(\varphi\circ y)$ is compactly supported in $\Omega$, as $y$ is a homeomorphism and $\varphi$ is compactly supported in $\O^y$.
As a consequence, the relation \eqref{eq:p_yg-1} can be extended by continuity to 
all test functions  in the class $W^{1,p}(\O;\R^n)\cap C_{\rm c}^0(\O;\R^n)$
since $\bp_{y,g} \in\mathcal M(\O;\mathbb R^n)$.
 Therefore, $\varphi\circ y$ is an admissible test function for 
 equality  \eqref{eq:p_yg-1}.
 From  \eqref{eq:chain}, from the validity \eqref{eq:p_yg-1} and from the fact that  \eqref{eq:p_yg-1} holds with test functions in $W^{1,p}(\O;\R^n)\cap C_{\rm c}^0(\O;\R^n)$ we obtain
\begin{equation}\label{plus}
	|\nabla(g\circ y^{-1})|(\Omega^y)=\sup\left\{\int_{\O}(\varphi\circ y)\cdot \d 
\bp_{y,g}\;|\;\varphi\in C^\infty_{\rm c}(\O^y;\R^n),\;\|\varphi\|_\infty\leq 1\right\}.
\end{equation}
The definition of total variation and \eqref{plus} directly imply
$
|\nabla(g\circ y^{-1})|(\Omega^y)	\leq |\bp_{y,g}|(\O)
$.
\end{proof}

\section{Convergence of the phases}\label{sec:convergenceofphases}

\EM From here and through the rest of the paper, $\Omega$ is a bounded open Lipschitz set.
\UUU In this section, we prepare some tools which will later be used in the limit
passages in Sections \ref{sec:liminf} and \ref{sec:main}. \EEE

\begin{lemma}\label{equicontinuouslemma} Let $p\ge n$.
Let $(y_k)_k\subset W^{1,p}(\Omega;\mathbb R^n)$, $y\in W^{1,p}(\Omega;\mathbb R^n)$ be  homeomorphisms of finite distorsion such that $y_k\to y$ weakly in $W^{1,p}(\Omega;\mathbb R^n)$.
\begin{itemize}
\item[i)]
 If $A\subset\subset \Omega^y$, then there exists $k_0\in\mathbb N$ such that $A\subset\Omega^{y_k}$ for any $k>k_0$. 
\item[ii)] Assuming in addition that the sequence $(\det\nabla y_k)_k$ is equiintegrable on $\Omega$, there holds  $\lim_{k\to\infty}|\Omega^y\Delta \Omega^{y_k}|=0$. \KKK
\end{itemize}
\end{lemma}

\begin{proof} 
	i) First we prove that the sequence $y_k $ is uniformly converging on any  compact subset  $K\subset\subset \O$. From \cite[Theorem 1.3]{HeiK} we deduce that there exists a constant $ C(K,n)$ such that, for any $k$,
	$$
	\forall x_1,x_2 \in K\quad |y_k(x_1)-y_k(x_2)|\leq C(K,n)\|\nabla y_k \|_{L^n(\O)}\, \theta(|x_1-x_2|), \quad \theta(t):=|\ln(2/t)|^{-1/n}.
	$$
	Since $\|\nabla y_k\|_{L^n(\Omega)}$ is bounded, we obtain the equicontinuity of the sequence $(y_k)$ over $K$.
	Moreover, by combining \UUU equicontinuity \EEE on compact domains  with the bound  $\| y_k\|_{L^1(\O)} <C$, the uniform boundedness of  $y_k$  on $K$ follows:
	$$
	\sup\{ |y_k(x)|  : k\in \mathbb N,\, x\in K\}<\infty.
	$$
	In fact, suppose by contradiction that there exist sequences $(k_\ell)_\ell$ and $(x_\ell)_\ell \subset K$ such that $|y_{k_\ell}(x_\ell)|\ge \ell $ for any $\ell\in\mathbb N$. Fix a $\delta>0$ such that $K+B_\delta(0)\subset K'\subset \subset \O $ for a  compact set $K'$. By the equicontinuity on $K'$, there exist $r\in (0,\delta)$ such that 
	$$
	\forall k\in\mathbb N\;\;\forall x\in K\;\;\forall x'\in B_r(x),\quad x'\in K' \quad\mathrm{and} \quad |y_k(x')-y_k(x)|<1.
	$$ 
	Therefore, $ 	\| y_{k_\ell}\|_{L^1}\ge\int_{B_r(x_\ell)}|y_{k_\ell}(x')|dx'\ge |B_r(0)|\, (\ell-1)\to \infty$ for $\ell\to \infty$, 
	a contradiction.
	\EEE
By Ascoli-Arzel\`a
theorem, $y_k\to y$ uniformly on any compact subset of $\Omega$.

Let $S$ be such that $A\subset\subset S \subset\subset \Omega^y$. Since $y,y_k$ are homeomorphisms and there is uniform convergence of $y_k$ to $y$ on compact subsets of $\Omega$,  it is easy to conclude. Indeed, let $\eps:=\mathrm{dist}(\overline A, \partial S)$ so that $\eps>0$.
 Let
$U=y^{-1}(S)$ so that $U\subset\subset\O$ as $y$ is a homeomorphism. Let   $S_k=y_k(U)$. Since $y,y^k\in\bbY$
are homeomorphisms on $\overline{U}$, we have $\partial S=y(\partial
U)$ and $\partial S_k=y_k(\partial U)$.  By the above result we have
 $y_k\to y$ uniformly on $\overline U$,   thus fixing $\delta<\eps/2$ we get
$\sup_{x\in \overline U}|y(x)-y_k(x)|<\delta$  for $k$ large enough.   Hence, for
any boundary point $\xi\in\partial S_k$,   we have that 
$d(\xi,\partial S)<\delta$  for $k$ large enough. 
Since $d(\overline{A},\partial S)=\eps>2\delta$, we obtain $d(\overline A,\partial S_k)>\delta$, hence $\overline A\subset S_k\subset \Omega^{y_k}$ for any large enough $k$.

ii)  We have $\det \nabla y_k\to\det\nabla y$ weakly in $L^1(\Omega)$ as $k\to+\infty$. This follows from the boundedness in $L^p(\Omega;\mathbb R^n)$ of the sequence $(\nabla y_k)_k$ if $p>n$ and from the additional equiintegrability assumption if $p=n$. \KKK
Then, the property $\lim_{k\to\infty}|\Omega^y\Delta \Omega^{y_k}|=0$  is a consequence of  \cite[Theorem 4.4]{GP}. Indeed, the measure-theoretic images appearing in \cite{GP} are the usual images for our mappings that have a continuous representative.
\end{proof}
%
%
\begin{lemma}\label{equilemma} Let $p\ge n$ and $q>n-1$.
Let $(y_k)_k\subset W^{1,p}(\Omega;\mathbb R^n)$, $y\in W^{1,p}(\Omega;\mathbb R^n)$ be  homeomorphisms of finite distorsion  such that $y_k\to y$ weakly in $W^{1,p}(\Omega;\mathbb R^n)$ and $\sup_{k\in\mathbb N}\|K_{y_k}\|_{L^q(\Omega)}<+\infty$.
 Suppose that the sequence $(\det\nabla y_k)_k$ is equiintegrable on $\Omega$. \KKK
 Then 
$
	|y^{-1}(O_k)|\to |\O|$ and $|y_k^{-1}(O_k)|\to |\O|
	$
as $k\to\infty$, where $O^k:=\Omega^y\cap \Omega^{y_k}$.
\end{lemma}

\begin{proof}
%
%
	We preliminarily observe that $\nabla y^{-1}\in L^n(\Omega^y;\mathbb R^n)$ and $\nabla y_k^{-1}\in L^n(\Omega^{y_k};\mathbb R^n)$ for any $k\in\mathbb N$. This follows from the $L^q$-integrability of the distorsion, see \cite{HKM}. In particular, $\det\nabla y^{-1}\in L^1(\Omega^y)$ and $\det\nabla y_k^{-1}\in L^1(\Omega^{y_k})$. 
	Moreover, since $y$, $y_k$ are homeomorphisms, we have $\det\nabla y>0$ a.e. in $\Omega^y$ and $\det\nabla y_k^{-1}>0$ a.e. in $\Omega^{y_k}$ for any $k\in \mathbb N$ and then $y, y_k$ satisfy the Lusin condition $N^{-1}$, see \cite[Theorem 4.13]{HK}. In particular, $y^{-1}$ satisfies the Lusin condition $N$ so that the area formula holds (with equality) and entails
	\begin{equation}\label{twoareas}
	|\Omega|=|y^{-1}(\Omega^y)|=\int_{\Omega^y}\det\nabla y^{-1}\,d\xi,\qquad|y^{-1}(O^k)|=\int_{O^k}\det\nabla y^{-1}\,d\xi,\quad k\in\mathbb N.
	\end{equation}
	Since  $|\O^y\setminus O^k|\to 0$ by Lemma \ref{equicontinuouslemma},
	we get from \eqref{twoareas} as $k\to\infty$
	$$
	|y^{-1}(O_k)|=\int_{O^k}\det\nabla y^{-1}\,\d\xi\to \int_{\O^y}\det\nabla y^{-1}\,\d\xi=|\O|.
	$$
	With the same change of variables for $y_k^{-1}$ that satisfies the Lusin condition $N$ we get
	\begin{align*}
	|y_k^{-1}(O_k)|&=\int_{O^k}\det\nabla y_k^{-1}\,\d\xi=\int_{\O^{y_k}}\det \nabla y_k^{-1}\,\d\xi-\int_{\O^{y_k}\setminus O^k}\det\nabla y_k^{-1}\,\d\xi\\&
	=|\O|-\int_{\O^{y_k}\setminus \O^y}\det\nabla y_k^{-1}\,\d\xi.
	\end{align*}
	Thanks to the results in \cite{OT},
	the uniform bound on $\|K_{y_k}\|_{L^q(\Omega)}$ yields the equi-integrability of the family $(\det\nabla y_k^{-1})_{k}$, as proven in \cite[Lemma 5.1]{GKMS}.
	Since  $ |\O^{y_k}\setminus \O^y|\to 0$ as $k\to\infty$ by Lemma \ref{equicontinuouslemma}, the statement follows.
\end{proof}

\begin{lemma}[Convergence of the reference phases]\label{lem:1}
Let $p\ge n$, $q>n-1$.
Suppose that
\begin{itemize}
\item[i)] $(y_k)_k\subset W^{1,p}(\Omega;\mathbb R^n)$, $y\in W^{1,p}(\Omega;\mathbb R^n)$ are homeomorphisms of finite distorsion such that $y_k\to y$ weakly in $W^{1,p}(\Omega;\mathbb R^n)$ as $k\to\infty$,
\item[ii)]$\sup_{k\in\mathbb N}\|K_{y_k}\|_{L^q(\Omega)}<+\infty$  and the sequence $(\det\nabla y_k)_k$ is equiintegrable on $\Omega$, \KKK
\item[iii)] $(\zeta_k)_k\subset L^\infty(\Omega^{y_k};\R^h\KKK)$, $\zeta\in L^\infty(\Omega^y)$ and $\|\zeta_k-\zeta\|_{L^1(\Omega^y\cap\Omega^{y_k})}\to 0 $ as $k\to\infty$,
\item[iv)] $|\zeta_k(\xi)|\le M$ holds a.e. in $\Omega^{y_k}$, for any $k\in\mathbb N$.
\end{itemize}	
Then $K_y\in L^q(\Omega)$, $|\zeta(\xi)|\le M$ a.e. in $\Omega^y$ and
$
\|\zeta_k\circ y_k-\zeta\circ y\|_{L^1(\Omega)}\to 0
$ as $k\to\infty$. 
	\end{lemma}
\begin{proof}
We use the notations $O^k:=\Omega^y\cap\Omega^{y_k}$ and $E_k:=y^{-1}(O^k)\cap y_k^{-1}(O^k)$. 
Since $y^{-1}(O_k)\subset \O $ and  $ y_k^{-1}(O_k)\subset \O$,
 in order to prove that $|\Omega\setminus E_k|\to 0$ as $k\to\infty$  it is sufficient to show
	$
	|y^{-1}(O_k)|\to |\O|$ and $ |y_k^{-1}(O_k)|\to |\O|,
	$
which are in turn proven in Lemma \ref{equilemma}.

Let us  prove that $|\zeta|\le M$ a.e. in $\Omega^y$. Indeed, suppose not and let $B:=\{\xi\in\Omega^y:|\zeta(\xi)|>M\}$ so that $|B|>0$. Then there exists $\eps>0$ and $B'\subset B$  such that $|B'|>|B|/2$ and $|\zeta|>M+\eps$ a.e. in $B'$. By assumption iv), this implies $|\zeta_k(\xi)-\zeta(\xi)|>\eps$ a.e. on $B'$ for any $k$. Let $A\subset\subset\Omega^y$ be an open set such that $|\Omega^y\setminus A|<|B|/4$, so that $|A\cap B'|>|B|/4$. Therefore, $\|\zeta_k-\zeta\|_{L^1(A)}\ge \|\zeta_k-\zeta\|_{L^1(A\cap B')}\ge \eps |B|/4$ for any $k$.
 On the other hand, for  $k$ large enough we have $A\subset\subset \Omega^{y_k}$ by Lemma \ref{equicontinuouslemma}, hence assumption iii) implies that $\|\zeta_k-\zeta\|_{L^1(A)}$ goes to zero as $k\to\infty$, a contradiction.
 
 As seen in the proof of Lemma \ref{equilemma}, we have $\det\nabla y>0$ a.e. in $\Omega^y$ and $\det\nabla y_k^{-1}>0$ a.e. in $\Omega^{y_k}$ for any $k\in \mathbb N$. Then, the property $K_y\in L^q(\Omega)$ follows by the polyconvexity of the optimal distorsion function on the set of matrices of positive determinant.
 
 Next we prove the convergence of reference phases $
\|\zeta_k\circ y_k-\zeta\circ y\|_{L^1(\Omega)}\to 0
$ as $k\to\infty$. We clearly bound such norm by $2M|\Omega\setminus E_k| +\|\zeta_k\circ y_k-\zeta\circ y\|_{L^1(E_k)}$, therefore we are reduced to prove
 that $\|\zeta_k\circ y_k-\zeta\circ y\|_{L^1(E_k)}$ goes to zero as $k\to\infty$.
 The argument is similar to the one of \cite[Lemma 5.3]{GKMS}. Indeed, there holds $\|\zeta_k\circ y_k-\zeta\circ y\|_{L^1(E_k)}\le I_k+J_k$, where
 \[
 I_k:=\|\zeta_k\circ y_k-\zeta\circ y_k\|_{L^1(E_k)},\qquad J_k:= \|\zeta\circ y_k-\zeta\circ y\|_{L^1(E_k)}.
 \]
 
 About $I_k$, 
  since $y_k^{-1}$ satisfies the Lusin condition $N$  we may  change  variables as done in the proof of Lemma \ref{equilemma}  and obtain
  \begin{equation}\label{ug}
 I_k=\int_{E_k^{y_k}}|\zeta_k(\xi)-\zeta(\xi)|\det\nabla y_k^{-1}(\xi)\,d\xi.
 \end{equation}
 We  fix a small value $\delta>0$, and since $E_k^{y_k}\subset O_k$, by \eqref{ug} we have
 \begin{equation}\label{ike}
 I_k\le \int_{O^k}\det\nabla y_k^{-1}|\zeta_k-\zeta|\,d\xi\le\delta \int_{O_k\setminus A_k(\delta)}\det\nabla y_k^{-1}\,d\xi+2M\int_{A_k(\delta)}\det\nabla y_k^{-1}\,d\xi,
 \end{equation}
 where $A_k(\delta):=\{\xi\in O^k:|\zeta_k(\xi)-\zeta(\xi)>\delta|\}$. Notice that $$\delta|A_k(\delta)|\le \int_{A_k(\delta)}|\zeta_k-\zeta|\,d\xi\le \int_{O_k}|\zeta_k-\zeta|\,d\xi,$$ so that assumption iii) yields $|A_k(\delta)|\to 0$ as $k\to\infty$. We deduce that
 \[
 \lim_{k\to\infty} \int_{A_k(\delta)}\det\nabla y_k^{-1}\,d\xi=0,
 \]
 thanks to the equi-integrability property of $\nabla y_k^{-1}$ from  \cite[Lemma 5.1]{GKMS}. Inserting this in \eqref{ike} we get
 \begin{equation*}\label{ikappa}
 \limsup_{k\to0}I_k\le \limsup_{k\to 0} \,\delta \int_{O_k\setminus A_k(\delta)}\det\nabla y_k^{-1}\,d\xi\le \delta |\Omega|,
 \end{equation*}
 where we changed back variables and used $y_k^{-1}(O_k\setminus A_k(\delta))\subset\Omega$. 
 
 Concerning $J_k$, Let $\bar\zeta_\delta$ be a continuous compactly supported function in $\Omega^y$ such that $|\bar\zeta_\delta|\le M$ and such that $|\bar A_\delta|<\delta$, where $\bar A_\delta:=\{\xi\in \Omega^y:|\bar\zeta_\delta(\xi)-\zeta(\xi)|>\delta\}$. For instance, we may may take a mollification of the restriction of $\zeta$ to a large enough open set compactly contained in $\Omega^y$. We write $J_k=J_k^{(1)}+J_k^{(2)}+J_k^{(3)}$, where
 \begin{align*}
	J_k^{(1)}=\|\z\circ y_k-\overline{\z}_\delta\circ y_k\|_{L^1(E_k)},
	J_k^{(2)}=\|\overline{\z}_\delta\circ y_k-\overline{\z}_\delta\circ y\|_{L^1(E_k)},
	J_k^{(3)}=\|\overline{\z}_\delta\circ y-\z\circ y\|_{L^1(E_k)}.
	\end{align*}
Here, $J_k^{(1)}$ and $J_k^{(3)}$ can be treated exactly as $I_k$ by change of variables, and with the help of assumption iv) we have for any $k\in\mathbb N$
\begin{equation}\label{unoetre}
J_k^{(1)}\le \delta|\Omega|+2M\int_{\bar A_\delta}\det\nabla y_k^{-1}\,d\xi,\quad J_k^{(3)}\le \delta|\Omega|+2M\int_{\bar A_\delta}\det\nabla y^{-1}\,d\xi.
\end{equation}
 On the other hand, if $A\subset\subset \Omega$ is an open set such that $|\Omega\setminus A|<\delta$, we have
  \begin{equation}\label{addue}
 J_k^{(2)}\le 2M\delta+\int_A|\bar\zeta_\delta\circ y_k-\bar\zeta_\delta\circ y|\,dx\le 2M\delta+|\Omega|\,\sup_{x\in A}\,\omega_\delta(|y_k(x)-y(x)|)
 \end{equation}
 where $\omega_\delta$ is the modulus of continuity of $\bar\zeta_\delta$. Taking the limit as $k\to\infty$ in \eqref{addue}, since by Lemma \ref{equicontinuouslemma} we have uniform convergence of $y_k$ to $y$ in $A$, we get
 \begin{equation}\label{finaldue}
 \limsup_{k\to\infty} J_k^{(2)}\le 2M\delta.
 \end{equation} 
 By \eqref{unoetre}, \eqref{finaldue}, the equi-integrability of $\det\nabla y_k^{-1}$, the integrability of $\det\nabla y^{-1}$, by $|\bar A_\delta|<\delta$ and  the arbitrariness of $\delta$, we conclude that $J_k\to0$ as $k\to\infty$.
 \end{proof}

\section{\MK  lower bound \EOR }\label{sec:liminf}
%
%
\UUU This section collects some lower semicontinuity arguments, leading to the proof of the $\Gamma$-$\liminf$
inequality, namely Proposition
\ref{theo:liminf}. We start by a lower semicontinuity property of interfacial measures (see Definition \ref{ifmeas}). 
\EEE 

\begin{proposition}[Double lower semicontinuity of $\bp_{y,g} $]\label{lsemicontinuity}
 Let $(y_k)_k\subset W^{1,p}(\Omega;\mathbb R^n)$, $y\in W^{1,p}(\Omega;\mathbb R^n)$ be  homeomorphisms of finite distorsion \KKK
  such that $y_k\to y$ weakly in $W^{1,p}(\O;\mathbb R^n)$,
for $p\ge n$.  Let $(g_k)_k\subset  L^r_{loc}(\Omega)$, $g\in L^r_{loc}(\Omega)$, $r\in[\tfrac{p}{p-n},+\infty)$, be such that $g_k\to g $ strongly in $L^r_{loc}(\Omega)$.   If $\liminf_{k\to+\infty}|p_{y_k,g_k}|(\O)<\infty$, then there exists  $ \bp_{y,g}\in\mathcal M(\O;\mathbb R^n)$ satisfying \eqref{eq:p_yg-1} and
\begin{equation}\label{2min}
|\bp_{y,g}|(\O)\leq \liminf_{k\to+\infty}|p_{y_k,g_k}|(\O).
\end{equation}
\end{proposition}
\begin{proof}
Since $\nabla y_k\to
\nabla y$ weakly in  $L^p(\O)$, the convergence 
$
\cof\nabla y_k\to\cof \nabla y
$  holds  weakly in  ${L^{p/(n-1)}}(\O)$. 
Therefore, for any test function $\psi\in C^\infty_{\rm c}(\O;\R^3)
$, as $k\to\infty$ we have, 
$$
\int_{\O}\psi\cdot \d \bp_{y_k,g_k}=\int_{\O}g_k\cof\nabla y_k:\nabla\psi\,\d x\to \int_{\O}g\,\cof\nabla y:\nabla\psi\,\d x=: \bp_{y,g}(\psi),
$$
by weak-times-strong convergence; the last equality is a definition of the distribution on
the right side. By the lower semicontinuity of the total
variation,  we have that $\bp_{y,g}\in\mathcal M(\O;\mathbb R^n)$ and \eqref{2min} holds. 
\end{proof}

\begin{lemma}[Lower semicontinuity of bulk energy] \label{lem:lsc-Fbulk} 
	 Let assumptions \eqref{lscassumption} and \eqref{W-ass2} hold. Let $R>0$, let $(y,\zeta)\in\widetilde{\mathbb Q}^R$ and let the sequence $(y_k,\zeta_k)_k\subset\widetilde\bbQ^R$ be such that \KKK
	\begin{itemize}
		\item[i)] $y_k\to y$ weakly in $W^{1,p}(\O;\Rz^n)$, 
		\item [ii)]  $\lim_{k\to+\infty}\|\zeta_k- {\zeta}\|_{L^1(O^k)}=0 $, with $ O^k:=\O^{y_k}\cap \O^y$.
	\end{itemize}
	Then,
	$\;\;
		\mathcal{F}^{\,\rm bulk}(y,\z)\le\displaystyle 	\liminf_{k\to\infty}  \mathcal{F}^{\,\rm bulk}(y_k,\z_k).
	$ 
\end{lemma}
\begin{proof}
\EM We may assume that along a not relabeled subsequence $\sup_{k\in\mathbb N}  \mathcal{F}^{\,\rm bulk}(y_k,\z_k)<+\infty$. Thanks to the coercivity assumption \eqref{W-ass2}, the hypotheses of Lemma \ref{lem:1} are satisfied. \KKK
	 Letting  $z_k:=\z_k\circ y_k$,
	 Lemma~\ref{lem:1}  entails  $z_k\to z=\z\circ y$ in $L^1(\O;\R^h\KKK)$. Now, write the bulk energy functional as a function of $z$:
	 	\begin{equation*}
	 \widetilde{\mathcal{F}}^{\,\rm bulk}(y,z):={\mathcal{F}}^{\,\rm bulk}(y,z\circ y^{-1})=\int_\Omega 
	 W(\nabla y(x),z(x))\,\d x,
	 \end{equation*}
	 Since $ W(\cdot,\cdot)$ is lower semicontinuous in $\mathbb R^{n\times n}\times \mathbb R^h $ and is  poly-convex in the first argument, we can apply the result \cite[Corollary 7.9]{fonseca-leoni}, getting
    $\liminf_{k\to\infty}
	 \widetilde{\mathcal{F}}^{\,\rm bulk}(y_k,z_k)\ge
	 \widetilde{\mathcal{F}}^{\,\rm bulk}(y,z),$
	 which proves the claim. 
\end{proof}
\EEE

In the following, \EM we recall that $\Phi:\mathbb R^h\to\mathbb R^+$ is a continuous potential that vanishes only at  the points of $P$, and that \eqref{dcoeff} holds.
We take advantage of the following inequality, for a proof see \cite[Proposition 2.1]{B}. 

\begin{proposition}\label{baldo2.1}
\EM For any $\alpha\in\{1,\ldots, m\}$,  let $\varphi_\alpha:\mathbb R^h\to\mathbb R$ be defined by $\varphi_\alpha(z):=d_\Phi(p_\alpha,z)$, where the $p_\alpha$'s are the zeros of $\Phi$. Let $u\in W^{1,2}(\Omega;\mathbb{R}^{  h })\cap L^{\infty}(\Omega;\mathbb{R}^h)$. 
\EM
Then $\varphi_\alpha\circ u \in W^{1,2}(\Omega)$ and for any open set $A\subseteq \Omega$ there \KKK holds 
\[
\int_A |\nabla(\varphi_\alpha\circ u)|\le\int_ A \sqrt{\KKK\Phi}\circ u\, |\nabla u|.
\]
\end{proposition}

Before stating the liminf inequality, we recall that for a collection $\{\mu_\alpha\}_{\alpha=1,\ldots, m}$ of positive Borel measures on $\Omega$, the supremum measure is defined on open sets $A\subseteq \Omega$ as
\begin{equation}\label{lub}
\left(\bigvee_{\alpha=1}^m\mu\right)(A):=\sup\left\{\sum_{\alpha=1}^k\mu_\alpha(A_\alpha):
  \mbox{$(A_\alpha)$ pairw. disjoint open sets,  $\displaystyle\bigcup_{\alpha=1}^m A_\alpha=A$}\right\}
\end{equation}
Equivalently, the supremum measure is the smallest positive Borel measure $\nu$ such that $\nu(A)\ge \mu_\alpha(A)$ for any $\alpha\in\{1,\ldots, m\}$ and any open set $A\subseteq\Omega$.

\EM
The theory that we developed in Section \ref{sil2} shall play a crucial role in the liminf inequality. Indeed, as we will see through the next proof, as soon as $(y,\zeta)\in\widetilde {\mathbb Q}^R$ is a state with finite energy, i.e., $\mathcal F_\eps(y,\zeta)<+\infty$,  an interfacial measure exists for the couple $(\varphi_\alpha\circ\zeta\circ y,y)$, for any $\alpha=1,\ldots,m$. This is reminiscent of the notion of admissible states from \cite{S2,S3}, which are indeed defined as those couples of deformations and phase indicators that admit a suitable interfacial measure.

\KKK

\begin{proposition}[$\Gamma$-$\liminf$ inequality]\label{theo:liminf} Let $p\ge n$, $q>n-1$.
		Let $R>\max_{\alpha\in\{1,\ldots,m\}}|p_\alpha|$, where $p_1,\ldots, p_m$ are the zeroes of $\Phi$.  Let $(y,\zeta)\in\widetilde{\mathbb Q}^R$ and let  $(y_k,\zeta_k)_k,\subset\widetilde\bbQ^R$ be a sequence such that \KKK
	\begin{itemize}
		\item [i)] $ \liminf_{k\to+\infty}  \mathcal{F}^{\,\rm
                    int}_{\eps_k}(y_k,\z_k)<\infty$ for  some  vanishing sequence $(\eps_k)_k\subset(0,+\infty)$,
		\item[ii)] $y_k\to y$ weakly in $W^{1,p}(\O;\Rz^n)$, 
		\item [iii)]  $\lim_{k\to+\infty}\|\zeta_k- {\zeta}\|_{L^1(O^k)}=0 $, with $ O^k:=\O^{y_k}\cap \O^y$.
	\end{itemize}
	Then, there exist sets of finite perimeter  $E_\alpha^y\subset \O^y$, $\alpha=1,\ldots, m$ such that
	\begin{equation}\label{uno}
	{\z }=\sum_{\alpha=1}^m p_\alpha \chi_{E_\alpha^y}\end{equation}
	and
	\begin{equation}\label{due}
	 \frac12\sum_{\alpha,\beta=1}^m d_{\alpha,\beta}\mathcal H^{n-1}(E_{\alpha,\beta}^y)\leq \liminf_{k\to+\infty}  \mathcal{F}^{\,\rm int}_{\eps_k}(y_k,\z_k),
	\end{equation}
	where $E_{\alpha,\beta}^y:=\Omega^y\cap \partial^*E_\alpha^y\cap\partial^* E_\beta^y$.
	In particular,  one has that  $(y, {\z})\in\bbQ $.
\end{proposition}
\begin{proof}
Let $F\subset\subset\Omega^y$ be open. By Lemma \ref{equicontinuouslemma} we have $F\subset\Omega^{y_k}$ for any large enough $k$. Therefore, assumption i) and Fatou lemma imply
\[
\int_{F}\Phi(\zeta)\,d\xi\le\liminf_{k\to\infty}\int_F\Phi(\zeta_k)\,d\xi\le  \liminf_{k\to\infty}\, \eps_k\int_{\Omega^{y_k}}\frac1{\eps_k}\Phi(\zeta_k)\,d\xi\le\liminf_{k\to\infty}\eps_k\,\mathcal F_{\eps_k}^{\rm int}(y_k,\zeta_k)=0.
\]
The arbitrariness of $F$ and $\Phi\ge 0$ show that $\Phi(\zeta)=0$ a.e. in $\Omega^y$.

 \KKK
For any $\alpha\in\{1,\ldots, m\}$ and any open set $A\subseteq\Omega$ (so that $A^{y_k}$ is open as well, since $y_k$ is a homeomorphism) we have by Proposition \ref{baldo2.1}
\begin{align*}
	\int_{A^{y_k}}\Big( \frac{\eps_k}2|\nabla 
	\z_k|^2+\frac 1{\eps_k} \Phi(\z_k)\Big)\,\d\xi\geq \int_{A^{y_k}}\sqrt{2\Phi(\z_k)}\,|\nabla
	\z_k|\,\d\xi\ge
	\int_{A^{y_k}}|\nabla(\varphi_\alpha\circ\zeta_k)|\,d\xi.
	\end{align*}
	Therefore, 
	\begin{equation}\label{nosup}
	\int_{\O^{y_k}}\Big( \frac{\eps_k}2|\nabla 
	\z_k|^2+\frac 1{\eps_k} \Phi(\z_k)\Big)\,\d\xi\geq \int_{\O^{y_k}}\max_{\alpha=1,\ldots,m}|\nabla(\varphi_\alpha\circ \zeta_k)|\,d\xi = \left(\bigvee_{\alpha=1}^m |\nabla(\varphi_\alpha\circ \zeta_k)|\right)(\O^{y_k}).
	\end{equation}

We have $|\zeta_k|\le R$ and we let \KKK $z_k:=\z_k\circ y_k$, thus $z_k\in L^\infty(\O;\mathbb{R}^{h})$. We  clearly have $g^\alpha_k:=\varphi_\alpha\circ z_k\in L^\infty(\O)$, 
and since $\varphi_\alpha\circ \z_k=g^\alpha_k\circ y_k^{-1}$,  by invoking Theorem \ref{reverse} we see that \begin{equation}\label{zeresima}
|\nabla(\varphi_\alpha\circ\zeta_k)|(A^{y_k})=|(y_k)_\flat(\nabla(\varphi_\alpha\circ \zeta_k))|(A)=|p_{y_k, g^\alpha_k}|(A),
\end{equation}
for any open set $A\subseteq\Omega$, where $p_{y_k, g^\alpha_k}$ is an interfacial measure.
 By Lemma \ref{lem:1} we have $z_k\to z$ strongly in $L^1(\O;\mathbb{R}^{ h\EEE})$, hence $g_k^\alpha\to g^\alpha$  strongly $L^1(\O)$. As in the proof of Proposition \ref{lsemicontinuity}, we get the weak  convergence of measures $p_{y_k,g^\alpha_k}\rightharpoonup p_{y,g^\alpha}$, which yields lower semicontinuity for any open set $A\subseteq \Omega$, i.e.\EEE
\begin{equation}\label{prima}
|p_{y, g^\alpha}|(A)\le\liminf_{k\to\infty}|p_{y_k, g^\alpha_k}|(A).
\end{equation}
By defining $g^\alpha:=\varphi_\alpha\circ z$, still by Theorem \ref{reverse} we have
\begin{equation}\label{seconda}|p_{y, g^\alpha}|(A)=|\nabla(g^\alpha\circ y^{-1})|(A^y)=|\nabla(\varphi_\alpha\circ\zeta)|(A^y).
\end{equation} 
From \eqref{zeresima}, \eqref{prima} and \eqref{seconda} we get
\begin{equation*}\label{terza}
 |\nabla(\varphi_\alpha\circ\zeta)|(A^y) \le\liminf_{k\to\infty}  |\nabla(\varphi_\alpha\circ\zeta_k)|(A^{y_k})
\end{equation*}
for any open set $A\subseteq\Omega$ and any $\alpha\in\{1,\ldots, m\}$.
By the latter semicontinuity property and the definition \eqref{lub} of supremum measure, we obtain
\begin{equation}\label{terza-1}
\left(\bigvee_{\alpha=1}^m |\nabla(\varphi_\alpha\circ\zeta)|\right)(\Omega^y) \le\liminf_{k\to\infty}\left(\bigvee_{\alpha=1}^m  |\nabla(\varphi_\alpha\circ\zeta_k)|\right)(\O^{y_k}).
\end{equation}
In conclusion we obtain from \eqref{nosup} and \eqref{terza-1}
\begin{equation}\label{tre}
\liminf_{k\to\infty}\mathcal F_{\eps_k}^{\rm int}(y_k,\zeta_k)=\liminf_{k\to\infty} \int_{\Omega^{y_k}}\Big( \frac{\eps_k}2|\nabla 
	\z_k|^2+\frac 1{\eps_k} \Phi(\z_k)\Big)\,\d\xi\ge 
	\left(\bigvee_{\alpha=1}^m |\nabla(\varphi_\alpha\circ\zeta)|\right)(\Omega^y).
\end{equation}

In particular, $\varphi_\alpha\circ \zeta\in BV(\Omega^y)$ for any $\alpha\in\{1,\ldots,m\}$.  Since $\Phi(\zeta)=0$ a.e. in $\Omega^y$,  
 by invoking \cite[Proposition 2.2]{B} we get \eqref{uno} and 	\KKK
\[
\left(\bigvee_{\alpha=1}^m |\nabla(\varphi_\alpha\circ\zeta)|\right)(\Omega^y)=\frac12\sum_{\alpha,\beta=1}^m d_{\alpha,\beta}\mathcal H^{n-1}(E_{\alpha,\beta}^y).
\]
 Together with \eqref{tre}, this proves \eqref{due}.\KKK
\end{proof}
%

\section{Proof of the main results}\label{sec:main}
%
%

\UUU We are now in the position of providing a proof of
our main results, Theorems \ref{Prop:diffuse-ex}, \ref{th1}, and
\ref{th2}. \EEE

%
%

\begin{proofth0}
	Let  $(y_k,\z_k)_k\subset {\widetilde{\mathbb{Q}}}^R_{(y_0,\Gamma_0)}$ \KKK
	be a minimizing sequence  for functional  $\FF_\eps$, which is bounded from below due to \eqref{W-ass2}. 
	The coercivity of the potential $W$ from \eqref{W-ass2}  and the 
	generalized Friedrichs inequality imply that  one can extract
	a  not relabeled subsequence  such
	that $y_k\to y$ weakly  in $W^{1,p}(\O;\R^n)$.  The boundary condition
	is preserved in the
	limit. 
	We conclude \UUU by Lemma \ref{init} \EEE that $y\in \mathbb{Y}$ and $y=y_0$
	on $\Gamma_0$, \EM recalling that the assumption  on $y_0$ (not constant on $\Gamma_0$) prevents $y$ from being a constant map.
	\KKK  

	Denote by $\eta_k $ and $H_k$ the   zero  extensions on $\Rz^n$  of  $\z_k$ and $\nabla \z_k$ respectively.
	The coercivity  of $\FF^{\rm int}_\eps$   and the uniform bound $|\zeta_k|\le R$ \KKK imply that
	one can extract  not relabeled subsequences such that 
	$\eta_k\to \eta$ weakly* in $L^\infty(\R^n;\R^h\KKK)$   and $H_k\to H$
	weakly in $L^2(\R^{n};\R^{h\times h}\KKK)$.  Set now  $\z:=\eta|_{\O^y}$. 
		For every $\delta>0$, let $O_\delta:=\{\xi\in\O^y|\, {\rm
		dist}(\xi,\partial\O^y)>\delta\}\subset\subset\O^y$.   We have that  $\O^y=\cup_\delta O_\delta$
	and, by Lemma
	\ref{equicontinuouslemma}, $ O_\delta\subset \O^{y_k}$
	for  $k$  large. 
	For every $\xi_0\in O_\delta$ and $B(\xi_0,r)\subset O_\delta$ we
	have  that 
	$\eta_k\to \eta$ weakly in $W^{1,2}(B(\xi_0,r);\R^h\KKK)$. This implies
	that $H=\nabla\eta=\nabla\z $ almost everywhere in
	$B(\xi_0,r)$. Moreover,  by possibly extracting again, one
	has that   $\eta_k\to\eta $ strongly in
	$L^2(B(\xi_0,r);\R^h\KKK)$. 
	As every $\xi\in\O^y$ belongs to some $O_\delta$ for $\delta$
	small enough, we get that   $H=\nabla \z$ almost
	everywhere~in $\O^y$. Now, by the weak \MK lower semicontinuity \EOR of the $L^2$-norm
	\begin{align}
	\liminf_{k\to\infty}\int_{\O^{y_k}} |\nabla \z_k|^2\,\d\xi=\liminf_{k\to\infty}\int_{\R^n} |H_k|^2\,\d\xi\ge \int_{\R^n}  |H|^2\,\d\xi\ge\int_{\O^y} |\nabla \z|^2\,\d\xi.\label{prove}
	\end{align}

	The local strong convergence $\eta_k\to \eta$ in $L^2(B(\xi_0,r);\mathbb R^h\KKK)$ for any
	$B(\xi_0,r)\subset\subset\O^y$ and  $|\eta-\eta_k|\leq C$ imply the strong $L^2$-convergence on ${\O^y}$, hence, up to extracting again, the pointwise convergence to $\z$ on  $\O^y$,  and thus $|\zeta|\le R$. \KKK By the Fatou Lemma, we  find 
	\begin{align*}
\liminf_{k\to\infty}\int_{\O^{y_k}}  \Phi(\eta_k)\,\d\xi=\liminf_{k\to\infty}\int_{\R^n}   \Phi(\eta_k)\,\d\xi \ge \liminf_{k\to\infty}\int_{\O^y}   \Phi(\eta_k)\,\d\xi\ge\int_{\O^y}   \Phi(\z)\,\d\xi.\label{prove1}
	\end{align*}
	Thus we have proven the weak lower semicontinuity of the
	interfacial energy $\mathcal{F}^{\,\rm int}_\eps(y_k,\z_k)$.

	As for the bulk contribution, because of  the convergence  $\|\zeta-\zeta_k\|_{L^1(\O^{y_k}\cap \O^y)} \to 0$, we can apply Lemma \ref{lem:lsc-Fbulk} and obtain the lower semicontinuity of $\mathcal{F}^{\,\rm bulk}(y,\zeta) $.
    \EEE
	Together with \eqref{prove}, this proves  that $(y,\z)$ is
	a minimizer of $\FF_\eps$ on  $ {\widetilde{\mathbb{Q}}}^R_{(y_0,\Gamma_0)}$ \KKK by means of the direct method \cite{dacorogna}.    
\end{proofth0}

\begin{proofth1}
	Let $(y_k,\z_k)\in \bbQ_{(y_0,\Gamma_0)}$  be a minimizing sequence  for
	$\FF_0$.  As in  the proof of Theorem \ref{Prop:diffuse-ex}, we can assume, up to extraction of a  not
	relabeled  subsequence, that $y_k\to  y$ weakly in
	$W^{1,p} $  for some $y\in\mathbb Y$,  and the coercivity assumption \eqref{W-ass2} also implies that  $\det\nabla y_k$ are equiintegrable functions on $\Omega$. \KKK 
	
	Let  $F_k=(F_k^1,\ldots, F_k^m)$, with $F_k^\alpha=\{ \z_k=p_\alpha \}$, $\alpha=1,\ldots, m$, be the partition of $\O^{y_k}$ corresponding to a phase configuration $\zeta_k$; we can identify the sequence of states
	with the sequence $(y_k,F_k)_k$. Since  the interface energy  $$\sum_{\alpha,\beta=1}^m d_{\alpha,\beta}\mathcal H^{n-1}((F_k)_{\alpha,\beta}),$$
	where $(F_k)_{\alpha,\beta}:=\partial^*F_k^\alpha\cap \partial^*F_k^\beta\cap \O^{y_k}$,
	 is bounded along the
	sequence $(y_k,F_k)_k$, the sets  $F_k$ have uniformly  bounded
	perimeters,  namely,   $\per(F_k^\alpha,\O^{y_k})\le c $.

	For   $\ell\in \Nz$, let $O^{\ell}:=\{x\in\O^y|\, {\rm
		dist}(x,\partial\O^y)>2^{-\ell}\}\subset\subset\O^y$.  As
	$O^{\ell}\subset\O^{y_k}$ for $k$ large enough due to Lemma
	\ref{equicontinuouslemma}, for any given $\ell\in\Nz$  we have that $\limsup_k
	\per(F_k^\alpha,O^{\ell})\le c$ for any $\alpha$.  We can hence   find a measurable
	set $(G^\alpha)^{\ell}\subset  O^{\ell}$ and a  not relabeled  subsequence $F_{h}$ such that
	\[|(F_{h}^\alpha\Delta (G^\alpha)^{\ell})\cap O^{\ell}|\to 0\quad \textrm{for}\quad h\to\infty.
	\]
	For all  $\ell'>\ell$  we can  further  extract 
	a  subsequence $F_{h'}$ from  $F_{h}$ above  in
	such a way  that $|(F_{h'}^\alpha\Delta (G^\alpha)^{\ell'})\cap O^{\ell'}|\to 0 $
	and  $(G^\alpha)^{\ell'}\cap O^{\ell}=(G^\alpha)^{\ell} $.  From  the nested
	family of subsequences corresponding to $\ell=1,2,\ldots$  we
	extract by a diagonal argument a further subsequence $F_{k'}$. By
	setting     $F^\alpha:=\cup_\ell (G^\alpha)^{\ell}$  and,  owing to $
	O^{\ell}\nearrow\O^y$,  we get that  
	\[
	|(F_{k'}^\alpha\Delta F^\alpha)\cap \O^{y}|\to 0.
	\]
	Now, the sets $F^\alpha$ has finite perimeter in $\O^y$ as a consequence of
	Proposition  \ref{lsemicontinuity}.  By letting  $\z=\chi_F|_{\O^y} $  we
	then have that  $(y,\z)\in \mathbb Q_{(y_0,\Gamma_0)}$. \KKK 
	
	One is left to check that $\FF_0 (y,\z) \leq \liminf
	\FF_0(y_k,\z_k)$, which follows from the lower semicontinuity of
	$\FF_0$. Indeed, the lower semicontinuity of bulk part of $\FF_0$
	follows by the argument of  Lemma \ref{lem:lsc-Fbulk}. 
	 As concerns the interface term, the lower semicontinuity with respect to local convergence in measure is proven in \cite[Example 2.5]{AB}.
\end{proofth1}
%

%
In Proposition \ref{theo:liminf} a $\liminf$ inequality for the interface functional has been established. Combined with the lower semicontinuity of the bulk energy (Lemma \ref{lem:lsc-Fbulk}), we conclude that the whole energy functional  satisfies a $\Gamma-\liminf$ inequality w.r.t. the convergence notion of Lemma \ref{lem:lsc-Fbulk}. 
%
%
%
\MK Under the full Dirichlet conditions on the boundary of the domain, we shall prove Theorem \ref{th2} by using a  Modica-Mortola \cite{MM} recovery sequence deeply generalized by Baldo  in \cite{B}. \EOR 
 The $\Gamma$-convergence allows to prove  the convergence of the phase field solutions to the sharp interface solution.
\begin{proofth2} 
We first claim that, if $(y,{\z})\in \bbQ_{y_0}$, $\O^{y_0}\subset\R^n$ 
	being  a Lipschitz domain, and if we let $F=(F_1,\ldots,F_m)$ with $ F_\alpha=\{\xi\in\O^{y_0}:\z(\xi)=p_\alpha\} $, there exists a sequence $(\z_k)_k\subset W^{1,2}(\O^y;\mathbb R^h)$  such that $|\z_k|\le R$ for suitable $R>\max_{{\alpha\in\{1,\ldots m\}}}|p_\alpha|$ \EEE 
	and  such that 
	$$
	\lim_{k\to\infty}\|\z_k- {\z}\|_{L^1(\O^y)}=0 \quad\mbox{and}\quad \frac12\sum_{\alpha,\beta=1}^m d_{\alpha,\beta}\mathcal H^{n-1}(F_{\alpha,\beta})+\mathcal{F}^{\,\rm bulk}(y,\z)=\lim_{k\to\infty}\mathcal{F}_{\eps_k}(y,\z_k).	$$
	Indeed,
since the   $y$-component is a constant sequence, the claim completely  rests on the construction of the  recovery sequence $(\z_k)_k$ provided   by Baldo \cite{B}  (such a sequence is also satisfying $\int_{\Omega^y}\zeta_k(\xi)\,d\xi=\int_{\Omega^y}\zeta(\xi)\,d\xi$ for any $k\in\mathbb N$, thus justifying our observations in Remark \ref{lastremark}). In order to use this result, we need to assume the Lipschitz regularity of the deformed domain through the imposition of Dirichlet boundary conditions on the whole boundary of $\Omega$.  Moreover, by inspecting the construction of the recovery sequence from \cite[Section 3]{B}, we see that we can obtain a sequence $(\zeta_k)_k$ that is uniformly bounded, i.e., such that $|\zeta_k|\le R_0$ for large enough $R_0$ (only depending on the multiwell potential $\Phi$).  Then,
	since $\zeta_k\circ y\to \zeta\circ y$ in $L^1(\O;\mathbb R^h)$ follows by Lemma \ref{lem:1}, the convergence of the bulk part is obtained by dominated convergence 
	by
	means of assumptions \eqref{W-ass1} and \eqref{W-ass5}. \KKK
	The claim is proved. 

 The rest of the proof follows the one in \cite{GKMS}.
  Here we give a summary of it. 
 Let $k\in\mathbb N$, let $(y_k,\z_k)$ be a minimizer (provided by Theorem \ref{Prop:diffuse-ex}) for  $ \FF_{\eps_k} $ over $\widetilde{\mathbb Q}^R_{y_0}$, $R>R_0$,  and let $(y^*,\z^*)\in \bbQ_{y_0}$ be a state of finite energy for $\mathcal F_0$ whose recovery sequence is  $(y^*,\z_{k}^*)\subset\widetilde{ \bbQ}^R_{y_0}$. Using $\FF_{\e_k}(y_k,\z_k)\leq \FF_{\e_k}
 (y^*,\z_{k}^*) $ and the fact that $\FF^{}_{\e_k}(y^*,\z_{k}^*) \to  \FF^{}_0(y^*,\z^*)$ as $k\to\infty$, we conclude that $ \FF_{\e_k}(y_k,\z_k)\leq C$.
The coercivity 
 \eqref{W-ass2}  along with Friedrichs inequality ensures that 
 $y_k\to y$ weakly in $W^{1,p}(\O;\R^n)$ 
 for some not
 relabeled subsequence.  Moreover, $y\in\mathbb Y$ and $y=y_0$ on $\partial\Omega$ (hence, $\O^{y_k}=\O^y=\O^{y_0}$).
 The uniform
 bound on  $\mathcal{F}^{\rm
 	int}_{\eps_k}(y_k,\zeta_k)= \mathcal{F}^{\rm
 	int}_{\eps_k}(y,\zeta_k)$ also yields strong $L^1(\O^{y};\R^h)$ compactness for
 the sequence $\zeta_k$. This implies the existence of  $\zeta\in
 L^\infty(\Omega^y; \mathbb R^h)$,    such that $|\zeta|\le R$ and \KKK $\|\zeta_k-\zeta\|_{L^1(\O^y)}\to
 0$  for some  not relabeled subsequence. \KKK  By Proposition \ref{theo:liminf},  
 $\zeta$ is takes values in $P$ and 
 $$\FF^{\rm int}_0(y,\z) \leq \liminf_{k\to \infty} \FF^{\rm
 	int}_{\e_k}(y_k,\z_k).$$
 Now, we show that $(y,\z)$ is a minimizer
 $\FF_0$ on $\mathbb{Q}_{y_0}$.
 In fact, for any $(\tilde y,\tilde \z) \in
 \mathbb{Q}_{y_0}$, let $(\tilde y,\tilde
 \z_k)$ \MK be \EOR  its recovery sequence: $\FF_{\e_k}(\tilde y,\tilde
 \z_k) \to \FF_0(\tilde y, \tilde \z)$ as $k\to\infty$. By the lower semicontinuity
of the bulk term $\FF^{\rm bulk}$, 
 $$\FF_0(y,\z) \leq \liminf_{k\to \infty} \FF_{\e_k}(y_k,\z_k)\leq
 \liminf_{k\to \infty} \FF_{\e_k}(\tilde y,\tilde \z_k) = \FF_0(\tilde
 y, \tilde \z),$$
\UUU which proves the assertion. \EEE
\end{proofth2}

\section*{Acknowledgements}  
This research of M.K. was supported by the
FWF-GA\v{C}R project 19-29646L    and by the 
OeAD-M\v{S}MT  project   8J19AT013.
\EM E.M. acknowledges support from the MIUR-PRIN project No 2017TEXA3H. 
\KKK
U.S.\ \UUU is supported by Austrian Science Fund (FWF) projects
F\,65, W\,1245,  I\,4354, and P\,32788 and by the 
Vienna Science and Technology Fund (WWTF) project
MA14-009. \EEE


\begin{thebibliography}{99}




%


\bibitem{AB} Ambrosio, L.,  Braides, A.:
Functionals defined on partitions in sets of finite perimeter II:  Semicontinuity,  relaxation and homogenization.
{\em J. Math. Pures Appl.}, {\bf 69} (1990), 307--333 
\EEE



\bibitem{AFP} 
 Ambrosio, L.,   Fusco, N.,  Pallara, D.:
 {\em Functions of Bounded Variation and Free Discontinuity Problems}.
 Oxford mathematical monographs. Oxford University Press, Oxford,
  2000.

\bibitem{B}  Baldo, S.: 
Minimal interface criterion for phase transitions in mixtures of Cahn-Hilliard fluids
{\it Annales de l'I. H. P. section C}, {\bf 7} , 
 (1990),  67--90.
	
	\bibitem{Ball-1977}
Ball, J.M.:  Convexity conditions and existence theorems in nonlinear elasticity. {\em Arch. Ration. Mech. Anal.} {\bf 63} (1977), 337--403.


 %
 \bibitem{ball-crooks}
Ball, J.M.,  Crooks, E.C.M.: Local minimizers and planar interfaces in
a phase-transition model with interfacial energy.  {\it
  Calc.~Var. Partial Differential Equations}, {\bf 40} (2011), 501--538.
 %
 \bibitem{BCO}
 Ball, J.M., Currie, J.C.,  Olver, P.L.:
Null Lagrangians, weak continuity, and
variational problems of arbitrary order.
{\it J.\ Funct.\ Anal.}  {\bf 41} (1981), 135--174.

\bibitem{ball-james} 
Ball, J.M., James, R.D.: Fine phase mixtures as minimizers 
of energy. {\it Archive Ration. Mech. Anal.} {\bf 100} (1988),
13--52.


%
%
%

%
\bibitem{ball-mora}
 Ball, J.M., Mora-Corral, C.: A variational model allowing both smooth
 and sharp phase boundaries in solids.  {\it Comm. Pure  Appl.~Anal.} {\bf 8} (2009), 55--81.

%
%
%

\bibitem{bhattacharya}
 Bhattacharya, K.: {\it Microstructure of martensite. Why it forms and 
how it gives rise to the shape-memory effect.} Oxford University~Press, New York, 2003.

%
%
%
%
%
%
%
  
  \bibitem{ciarlet-necas}
 Ciarlet, P.G., Ne\v{c}as, J:  {Injectivity and
  self-contact in nonlinear elasticity}. {\it Arch. Ration. Mech. Anal.} {\bf 97} (1987), 171--188.
  
	
	
\bibitem{dacorogna}
 Dacorogna, B.: {\it Direct Methods in the Calculus of Variations.} 2nd.~ed.,
Springer, Berlin, 2008.



\bibitem{DalMaso93}
{Dal Maso}, G.:
\newblock {\em An introduction to {$\Gamma$}-convergence}.
\newblock Progress in Nonlinear Differential Equations and their Applications,
  8. Birkh\"auser Boston Inc., Boston, MA, 1993.







\bibitem{fonseca-leoni}
 Fonseca, I., Leoni, G.: {\it Modern Methods in the Calculus of
   Variations: $L^p$-Spaces.} Springer, New York, 2007.



\bibitem{F}
 Fonseca, I.: Interfacial energy and the Maxwell rule. 
{\it 
Arch. Ration. Mech. Anal.} {\bf 106} (1989) 63--95.







\bibitem{FMS}
Fusco, N., Moscariello, G., Sbordone, C.: The limit of $W^{ 1,1}$ homeomorphisms with finite
distortion, {\it Calc. Var. Partial Differential Equations} {\bf 33} (2008), 377--390.

\bibitem{GI}
 Gehring, F.,   Iwaniec, T.: The limit of mappings with finite distortion, {\it Ann. Acad. Sci. Fenn. A I} {\bf 24} (1999), 253--264.



\bibitem{GP}
 Giacomini, A., Ponsiglione, M.: Non-interpenetration of matter for SBV
deformations of hyperelastic brittle materials.
{\it Proc. Roy. Soc. Edinburgh Sect. A}, {\bf 138} (2008),  1019--1041.

\bibitem{GKMS} Grandi, G., Kru\v{z}\'ik, M., Mainini, E., Stefanelli, U., A phase-field approach to Eulerian interfacial energies. {\it Arch. Ration. Mech. Anal.} {\bf 234} (2019), 351--373. 

\UUU
\bibitem{Gurtin75}
Gurtin, M. E., Murdoch, A.: A continuum theory of elastic
material surfaces. {\it Arch. Ration. Mech. Anal.} {\bf 57} (1975), 291--323.
\EEE

\bibitem{G}
Gurtin, M. E.: On phase transitions with bulk, interfacial, and boundary energy,
{\it Arch. Ration. Mech. Anal.} {\bf 96} (1986), 243--264.


	
	\bibitem{HeiK}
	Heinonen, J., Koskela, P.: Sobolev mappings with integrable dilatations. {\it Arch. Rational Mech. Anal.} {\bf 125} (1993), 81--97.
	
	
\bibitem{HK}
 Hencl, S.,   Koskela, P.: {\it Lectures on mappings of finite
   distortion}. Lecture Notes in Mathematics, Vol. 2096, Springer, 2014.

\bibitem{HKM}
  Hencl, S., Koskela, P.,  Mal\' y, J.:
	Regularity of the inverse of a Sobolev homeomorphism in
        space. {\it Proc. Roy. Soc. Edinburgh Sect. A}
{\bf 136A} (2006), 1267--1285.

\UUU 
\bibitem{Javili}
Javili, A., McBride, A., Steinmann, P.: Thermomechanics of solids with
lower-dimensional energetics: On the importance of surface, interface,
and curve structures at the nanoscale. A unifying review. {\it
  Appl. Mech. Rev.} {\bf 65} (2013), 010802  (31 pages).
\EEE


\bibitem{KM}
Kohn, R. V., M\"uller, S.:
Surface energy and microstructure in coherent phase transitions.
{\it Comm. Pure Appl. Math.} {\bf 47}(4) (1994), 405--435.

%

%



 
%
\UUU
\bibitem{Levitas10}
Levitas, V.I., Javanbakht, M.:
 Surface tension and energy in multivariant martensitic transformations: phase-field theory, simulations, and model of
coherent interface. {\it Phys. Rev. Lett.} {\bf 105} (2010), 165701.

\bibitem{Levitas14}
Levitas, V.I.:
Phase field approach to martensitic phase transformations
with large strains and interface stresses. {\it
  J. Mech. Phys. Solids}, {\bf  70} (2014), 154--189.


\bibitem{Levitas14b}
Levitas, V.I., Warren, J.A.:
Phase field approach with anisotropic interface energy
and interface stresses: Large strain formulation. {\it
  J. Mech. Phys. Solids}, {\bf 91} (2016), 94--125.

\EEE



\bibitem{M}
Modica, L.: The gradient theory of phase transitions and the minimal interface
criterion, {\it Arch. Rat. Mech.  Anal.} {\bf 98} (1987),  123--142.

\bibitem{MM}
 Modica, L.,    Mortola, S.: Un esempio di $\Gamma$-convergenza. (Italian).
{\it Boll.
Un. Mat. Ital. B},
{\bf 14} (1977),  285--299.

\bibitem{Mu}
M\"uller S.: Higher integrability of determinants and weak convergence in $L^1$, {\it J. Reine Angew. Math.} {\bf 412} (1990), 20--34.

\EM
\bibitem{Mu2}
 M\"uller, S.: Variational models for microstructure and phase transitions, in “Calculus of Variations and Geometric Evolution Problems” (eds. S. Hildebrandt and M. Struwe), Springer- Verlag (1999), 85--210.
\KKK


\bibitem{OT}
 Onninen, J. ,   Tengvall, V.: Mappings of $L^p$-integrable distortion: regularity of the inverse.
{\it 
Proc. Roy. Soc. Edinburgh Sect. A}, {\bf 146} (2016), 647--663.


\bibitem{P}
Parry, G. P.:  On shear bands in unloaded crystals. {\it J. Mech. Phys. Solids} {\bf 35} (1987), 367--382.

\bibitem{reshetnyak}  Reshetnyak, Y. G.:
 Some geometrical properties of functions and mappings with
generalized derivatives.
{\it Sibirsk. Math. Zh.} {\bf  7} (1966), 886--919.



%
%
%
%

%

\bibitem{S2}
\v Silhav\'y, M.: Phase transitions with interfacial energy: interface null Lagrangians, polyconvexity, and
existence. In: Hackl, K. (ed.) {\it IUTAM Symposium on Variational Concepts with Applications to the
Mechanics of Materials}, pp. 233--244. Springer, Dordrecht (2010).

\bibitem{S3}
 \v Silhav\'y, M.: Equilibrium of phases with interfacial energy: A
 variational approach. {\it J. Elast.} {\bf  105} (2011), 271--303.

\bibitem{S}
Sternberg, P.: The effect of a singular perturbation on nonconvex variational
problems. {\it Arch. Ration. Mech.  Anal.} {\bf  101} (1988),  209--260.



%












	
	
\end{thebibliography}
\end{document}